\documentclass[11pt,a4paper]{amsart}

\usepackage{amsmath}
\usepackage{amssymb}
\usepackage{amsfonts}
\usepackage{amsthm}
\usepackage{amssymb}
\usepackage{amsbsy}
\usepackage{mathtools}
\usepackage{color}
\usepackage{xcolor}
\usepackage{graphicx}
\usepackage{stackengine}
\usepackage{subfigure}
\usepackage[utf8]{inputenc}
\usepackage{csquotes}
\usepackage[greek,english]{babel}
\usepackage{fancyhdr}
\usepackage{comment}

\usepackage{enumitem}
\RequirePackage[normalem]{ulem}
\usepackage[makeroom]{cancel}

\usepackage[bookmarks]{hyperref}

\newtheorem{theorem}{Theorem}[section]
\newtheorem{corollary}[theorem]{Corollary}
\newtheorem{lemma}[theorem]{Lemma}
\newtheorem{proposition}[theorem]{Proposition}
\newtheorem{assumption}[theorem]{Assumption}

\theoremstyle{definition}
\newtheorem{definition}[theorem]{Definition}
\theoremstyle{remark}
\newtheorem{remark}[theorem]{\textbf{Remark}}
\newtheorem{example}[theorem]{\textbf{Example}}
\newcommand{\E}{\mathbb{E}}

\renewcommand{\P}{\mathbb{P}}
\newcommand{\Q}{\mathbb{Q}}

\newcommand{\df}{\mathrm{d}}
\newcommand{\di}{\mathrm{d}}
\newcommand{\wti}{\widetilde}

\newcommand{\pian}[2]{\dfrac{\partial #1}{\partial #2}}
\newcommand{\piann}[2]{\dfrac{\partial^{2}#1}{\partial #2^{2}}}

\newcommand{\R}{\mathbb{R}}

\newcommand{\Lc}{\mathcal{L}}

\newcommand{\Qb}{\mathbb{Q}}
\newcommand{\as}{a.s.}

\newcommand{\eg}{e.g.}
\newcommand{\cf}{c.f.}
\newcommand{\dx}{\mathrm{d}x}
\newcommand{\dy}{\mathrm{d}y}
\newcommand{\ds}{\mathrm{d}s}

\newcommand{\dt}{\mathrm{d}t}
\newcommand{\dz}{\mathrm{d}z}

\newcommand{\F}{\mathcal{F}}
\newcommand{\Fb}{\mathbb{F}}
\newcommand{\Fc}{\mathcal{F}}

\newcommand{\Tc}{\mathcal{T}}

\newcommand{\Uc}{\mathcal{U}}

\newcommand{\ind}{\boldsymbol{1}}

\DeclareMathOperator{\supp}{supp}

\newcommand{\Nb}{\mathbb{N}}
\newcommand{\Zb}{\mathbb{Z}}
\newcommand{\Tb}{\mathbb{T}}

\newcommand{\Pc}{\mathcal{P}}
\newcommand{\Ac}{\mathcal{A}}

\newcommand{\Wc}{\mathcal{W}}

\definecolor{orange}{rgb}{1,0.3,0.2}

\DeclareMathOperator*{\argmax}{arg\,max}

\title[Measure-valued HJB perspective on optimal adaptive control]{A measure-valued HJB perspective on Bayesian optimal adaptive control}
\author{Alexander M.~G.~Cox}
\thanks{Alexander M.~G.~Cox, Department of Mathematical Sciences, University of Bath, Bath, U.~K..\\ e-mail: \texttt{a.m.g.cox@bath.ac.uk}
}
\author{Sigrid K{\"a}llblad}
\thanks{Sigrid K{\"a}llblad, Department of Mathematics, KTH Royal Institute of Technology, Stockholm, Sweden. e-mail: \texttt{sigrid.kallblad@math.kth.se}.
The author gratefully acknowledges financial support from the Swedish Research Council (VR)
under grant 2020-03449.
}
\author{Chaorui Wang}
\thanks{Chaorui Wang, Department of Mathematical Sciences, University of Bath, Bath, U.~K..\\ e-mail: \texttt{cw2581@bath.ac.uk}.
The author is supported by a scholarship from the EPSRC Centre for Doctoral Training in Statistical Applied Mathematics at Bath (SAMBa) under the project EP/S022945/1 and the Chinese Council Scholarship.
}

\date{\today}

\begin{document}

\begin{abstract}
  We consider a Bayesian adaptive optimal stochastic control problem where a hidden static signal has a \emph{non-separable} influence on the drift of a noisy observation. Being allowed to control the specific form of this dependence, we aim at optimising a cost functional depending on the posterior distribution of the hidden signal. Our setup is in sharp contrast to existing work: we include costs that depend on the full posterior distribution in a form that admits a large class of non-linear relationships. Expressing the dynamics of this posterior distribution in the observation filtration, we embed our problem into a genuinely infinite-dimensional stochastic control problem using measure-valued martingales. We address this problem by use of viscosity theory and approximation arguments. Specifically, we show equivalence to a corresponding weak formulation, characterise the optimal value of the problem in terms of the unique continuous viscosity solution of an associated HJB equation, and construct a piecewise constant and arbitrarily-close-to-optimal control to our main problem of study. As a byproduct of our analysis, we also provide a novel stability result for a class of measure-valued SDEs which we believe is of independent interest.
\end{abstract}

\maketitle

%\begingroup
%  \hypersetup{hidelinks}
%  \tableofcontents
%\endgroup
%{\hypersetup{linkcolor=black}
%\tableofcontents
%}

\section{Introduction}

We consider a Bayesian adaptive optimal stochastic control problem in continuous time on an infinite horizon.
There is an underlying static signal which cannot be observed directly but only via a noisy observation. 
More precisely, the drift of this observation process has a non-linear dependence on the signal and we are allowed to control the specific form of this dependence.
We take a Bayesian view on the estimation problem and assume that the prior is known and subsequently update our beliefs. 
The aim is to optimise an objective depending both on the control and the posterior distribution of the hidden signal.
The typical interpretation is that we are searching for a hidden `signal' and based on what we've learned so far we may choose different `search actions' which might be more or less effective but at the same time more or less costly.

\subsection*{Problem formulation}

In mathematical terms, given an unobservable random variable $X$ with law $\mu_0$, the observation process is given by         
    \begin{align}\label{eq:intro_dynamics_observation}
        \di Y^u_t=h(u_t,X)\dt+\di W_t,
        \quad Y^u_0=0,
    \end{align}
  where $u$ is the chosen control and the function $h$ specifies the spatial dependence on the signal for different actions. For any given control, the observation process generates a filtration, say $\mathcal{Y}^u$, which is explicitly control dependent. The posterior distribution, also referred to as the filter, is given by
    \begin{align*}
        \pi^u_t=\mathrm{Law}(X|\mathcal{Y}^u_t).
    \end{align*}
Given a discount rate $\beta$ and a running cost $k$ depending on the control and the current state of the posterior distribution, the aim is to optimise an objective of the form
    \begin{align}\label{eq:intro_objective}
        \E\int_0^\infty e^{-\beta t}k(\pi^u_t,u_t)\dt;
    \end{align}
    the optimisation is performed over controls $u$ which are adapted to their own generated filtration $\mathcal{Y}^u$, that is, controls that rely only on the information on $X$ generated from observing $Y^u$. 

\subsection*{Results and methodology}

We address the above problem by use of dynamic programming arguments and stochastic control methods. To this end, we first express the dynamics of the posterior distribution in the observation filtration; specifically, for any bounded continuous function $\varphi$, writing $\pi^u_t(\varphi)$ for $\int\varphi(x)\pi^u_t(\dx)$ etc., we have that 
\begin{align}\label{eq:pi_intro}
    \di\pi^u_t(\varphi)
    =\big[\pi^u_t\big(\varphi(\cdot)h(u_t,\cdot)\big)-\pi^u_t\big(\varphi\big)\pi^u_t\big(h(u_t,\cdot)\big)\big]\di I^u_t,
    \;\;\pi^u_0=\mu_0,
\end{align}
where $I^u$ is a Brownian motion in the observation filtration. The filter can thus be viewed as a controlled measure-valued martingale (MVM) with dynamics driven by a Brownian motion.
We introduce two related formulations: A weak problem formulation where the optimisation is over weak solutions to \eqref{eq:pi_intro}, and an approximate formulation where one restricts to measure-valued processes solving \eqref{eq:pi_intro} for piecewise constant controls. A priori, it is not clear that either of these formulations coincide. 
However, our main result establishes that the value of our original problem, the value of the weak formulation, and the limiting value of the approximate formulation as the time grid is refined, are equal in value. 
\emph{A posteriori}, it is thus clear that the additional information available within the weak formulation does not alter the value of the problem. 
We also establish that the value of the problem can be characterised in terms of the unique continuous viscosity solution of an associated HJB equation.
Finally, we construct a piecewise constant and arbitrarily-close-to-optimal strategy for our main problem of interest.

Our proof bears resemblance to the Barles-Souganides approximation method introduced in \cite{barles_1991}, as well as to the stochastic Perron approach developed in \cite{bayraktar_2013}; it proceeds by establishing that the limiting and weak value functions are sub and super solutions, respectively, of an associated HJB equation, by sandwiching our main problem of interest in between the two, and by applying appropriate comparison results. 
Our method is notably inspired by the contributions in \cite{cohen_2025}. The key feature of our work is that we allow for a more general dependence on the signal in the drift of the observation process; while we restrict to an action space enforcing a certain polynomial structure, we crucially allow for infinite polynomials. In contrast to the problem studied in \cite{cohen_2025}, our problem can therefore \emph{not} be reduced to a finite-dimensional one but calls for a genuinely infinite-dimensional analysis where one works directly with the probability-measure--valued filters. We carry out such an analysis by relying on the viscosity theory developed for stochastic control problems with MVMs in \cite{cox_2024}. In particular, we develop Barles-Souganides type approximation arguments for this infinite-dimensional setup; we find those developments to be of independent interest and hope that they will become useful also in other contexts. In order to establish continuity of the value function, we also provide a result on the stability of controlled MVMs in terms of their initial condition; we find this result too to be of independent interest.

Our main problem of study is notably given in a strong formulation where the observation filtration has an explicit dependence on the control. At the same time, admissible controls are constrained to depend on the information provided by the observation process only. When aiming at picking a `good control' there is thus a trade-off present since such controls must balance improving the estimation and minimizing the cost functional.
The effect of considering strong formulations within stochastic control problems with unobservable components was recently reviewed in \cite{cohen_2025}. In particular, they illuminated the link between controlling the information flow and the notions of \emph{dual effect} (referring to the control's effect on higher order moments of the posterior distribution) %(referring to the interplay between the control's effect on the state and its influence on the estimation of unobservable components)
and the much-studied trade-off between \emph{exploration and exploitation} (referring to the knowledge of the unobservable component contra cost optimisation); we refer to that article for further details and references on the topic.
Following \cite{cohen_2025}, we notably introduce our problem by defining a class of pre-admissible controls and restricting, in turn, to the subset of controls which rely on the information provided by the observations only.
Within the literature on optimal control under partial information, controls that are thus restricted to depend on the observation process only are often referred to as `strict-sense' controls. This is in contrast to so-called `wide-sense' controls which are allowed to depend on a somewhat larger filtration; see \eg{} \cite{benes_1991,elkaroui_1988,fleming_1982}.
It is also common to define admissible controls on some reference probability space and let the controlled dynamics enter the problem via a change of measure (rendering the filtration independent of the chosen control); see \eg{} \cite{bandini_2018,bensoussan_1992,karatzas_1992}.
Our main result verifies that for the problem at hand, such variants of weak formulations would agree in value.

\subsection*{Related literature}

   Bayesian search problems have been studied from a more applied perspective in \eg{} \cite{assaf_1985,ding_2015,fox_2003,yu_2019}.
    This literature aside, there are essentially three streams of
    research to which our problem relates:
    \smallskip
    
    \noindent
     {\bf Bayesian stopping problems:} Motivated by aims similar to ours, closely related problems have been considered within an optimal stopping framework. For example, in \cite{ekstrom_2022}, the authors let an unobservable state influence the drift of the signal process and their aim is to estimate this state as accurately but also as quickly as possible. Considering a Bayesian framework, they formulate it as an optimal stopping problem which they subsequently address. An extended version of such a stopping problem was studied in \cite{ekstrom_2024} to which we refer for further references.

     \smallskip
     \noindent
    {\bf Control problems with hidden state:} Our formulation includes the problem of minimising, over admissible controls, an objective of the form\footnote{Notably there is still a trade-off present when choosing the control since it affects both the cost and the estimation of the unknown signal; as before, this is a consequence of the fact that we consider a strong formulation where the choice of the control influences the available information through the filtration.}
    \begin{align}\label{eq:ext_y}
    \E\int_0^\infty e^{-\beta t}\hat k(X,u_t)\dt
    \end{align}
    for some given cost function $\hat k$; this can be seen by defining $k(\mu,v)=\mu(\hat k(\cdot,v))$, for any probability measure $\mu$ and action $v$.
    Closely related problems have been studied previously in the literature:
    In \cite{benes_1991} and \cite{karatzas_1992}, the authors let the observation process follow the dynamics corresponding to letting $h(v,x)=vx$ in \eqref{eq:intro_dynamics_observation}, for real-valued controls, and the aim is to optimise an objective similar to \eqref{eq:ext_y} but with a running cost depending on the observation process itself; specifically, the cost takes the form $\hat k(Y^u_t)$ for some function $\hat k$ satisfying certain properties. The corresponding finite-horizon problem was studied in \cite{karatzas_1993}.
    More recently, the above-mentioned article \cite{cohen_2025}, considered an observation process whose dynamics is again similar to \eqref{eq:intro_dynamics_observation} but with $h(u_t,X)$ replaced by $b(t,Y^u_t,u_t)X$ for some function $b$ (they allowed for a multi-dimensional signal) and with a similar controlled volatility coefficient appearing in front of the noise term. They then optimised a finite-horizon objective with a running cost of the form $\hat k(t,Y^u_t,X,u_t)$.
    While these papers motivated our work, our setup is different: We crucially allow for a \emph{non-linear} dependence on the signal in the drift of the observation process and a \emph{non-linear} dependence on the posterior distribution in the cost functional. In order to render the presentation focused on the difficulties arising from this non-linearity, we develop the theory for the case where neither the drift nor the cost has any explicit dependence on the observation process.

    \smallskip
    \noindent
    {\bf Stochastic control under partial information:} Our problem also bears resemblance to a large class of problems known as \emph{stochastic control under partial information}. The difference between our setup and this class of problems is that we keep our signal constant while we control the observation rather than the signal itself. In comparison, in the partial observation literature, one is typically interested in a problem where the hidden signal, say $X_t$, $t\ge 0$, follows itself a controlled It{\^o} dynamics while the observation is given by \eqref{eq:intro_dynamics_observation} --- with $X$ replaced by $X_t$ --- for some non-controlled\footnote{Notably, \cite{bandini_2018}, \cite{knochenhauer_2024} and \cite{handel_2007} also allow for controlled observations.} function $h(x)$, $x\in\R$; the aim is to optimise an objective of the form \eqref{eq:ext_y} --- with $X$ replaced by $X_t$. The nomenclature `partially observable' refers to the fact that controls must be adapted to the filtration generated by the observation process; controls given in feedback form in terms of the signal are in general not admissible.
    
    The predominant approach has been to address this problem via a reformulation enabling viewing it as a control problem where the conditional distribution of the hidden signal plays the role of the controlled state variable.
    In its original form, this approach amounts to rewriting the objective onto the form \eqref{eq:intro_objective} where now $\pi^u_t=\mathrm{Law}(X_t|\mathcal{Y}^u_t)$; this filter then solves the so-called \emph{Kushner-Stratonovich (KS) equation} which is similar to \eqref{eq:pi_intro} but with additional (controlled) terms appearing due to the dynamics of $X$ itself. At the cost of the controlled process taking values in the space of probability measures, the resulting problem of optimising \eqref{eq:intro_objective} over solutions to the KS equation is conceptually simpler since it is a fully observable problem; for example, controls given in feedback form in terms of the filter are admissible.
    The strategy of linking the problem of original interest to such an observable control problem is often referred to as the \emph{separation principle}. Provided the latter problem admits an optimiser, the hope is to construct a solution to the original problem too based on that solution. When $X_0$ is Gaussian and the various coefficients appearing in the dynamics of the signal and observation processes are linear, the resulting filter, referred to as the \emph{Kalman-Bucy filter}, is Gaussian too and this problem can be solved explicitly (see e.g. \cite{fleming_1975}, \cite[Chapter~7]{bensoussan_1992} and \cite[Section~7.5]{handel_2007}). The study of the KS equation was also crucial for the analysis in \cite{fleming_1980} and \cite{elkaroui_1988} where the existence of optimal controls was established. %; see also \cite{christopeit_1980} and \cite{haussmann_1982}.
    %While offering a conceptually simple and elegant formulation, this problem is however difficult to solve in practice, and explicit results are few.
    
    A more common line of attack, however, has been to employ the separation principle in conjunction with the unnormalised filter, which is a measure-valued process $\sigma_t$ satisfying $\pi^u_t(\varphi)=\sigma^u_t(\varphi)/\sigma^u_t(1)$, $\varphi\in C_b$. By considering a change of measure rendering the observation process a Brownian motion, one may express the cost functional in terms of this unnormalised filter and also obtain an equation for it, the so-called \emph{Zakai equation}, which is a controlled measure-valued SDE driven by the observation process itself.
    One may again view the associated (separated) optimisation problem as a fully observable stochastic control problem featuring a controlled measure-valued process. Crucially, the Zakai equation turns out to be an easier equation in certain ways explaining the popularity of this method.
    For example, the study of the Zakai equation was a crucial component of the analysis in \cite{fleming_1982}. 
    More pertinently, it is often assumed that the unnormalised conditional distribution admits a density: the Zakai equation then reduces to the so-called \emph{Duncan-Mortensen-Zakai equation} for the density process itself. Within such a setup, there are numerous contributions to date:
    in \cite[Chapter 8]{bensoussan_1992}, the problem is dealt with by use of the stochastic maximum principle as well as via the DPP combined with semigroup arguments; 
    in \cite{hijab_III}, %(\cf{} \cite{hijab_I,hijab_II}) 
    it is proven that the value function is a viscosity solution (albeit in a rather weak sense) %(which does not enable getting comparison).
    to an associated HJB equation;
    and in \cite{lions_II} and \cite{gozzi_2000}, it is assumed that the density process belongs to $L^2$ and an weighted $L^2$-space, %(still a Hilbert space)
    respectively, and it is then shown that the value function is the unique viscosity solution to a corresponding HJB equation. %(see also \cite{lions_I,lions_III}).
    Motivated by similar questions, \cite{martini_2023} established an It\^o formula for solutions to the KS and Zakai equations (see also \cite{martini_2024} where viscosity solutions are discussed for some related equations).

In \cite{bandini_2018}, the authors consider yet a different implementation of the separation principle.
Introducing again the measure making the observations a Brownian motion, they consider the joint distribution under this measure, given the accumulated observations, of its (inverse) Radon-Nikodym derivative and the signal. 
%By changing measure and using the tower property, 
The objective can clearly be rewritten in terms of this probability-measure--valued process. To find an equation for its dynamics, one acknowledges that the Radon-Nikodym derivative solves an SDE driven by the observations. %; if necessary one also rewrites the dynamics for the signal in terms of the observation process.
Effectively, one is thus facing a situation analogous to the one of deriving the KS equation when the only thing being observed is one of the noise processes driving an `enlarged signal'; the associated normalised and unnormalised conditional probabilities then coincide and one easily obtains its dynamics. Again, one may then study the `separated problem' associated with this objective and controlled probability-measure--valued dynamics. In essence, one has obtained a simpler dynamics at the cost of introducing an additional factor. The authors consider a non-Markovian version of this (separated) problem which they address by use of BSDE methods. In a closely related work, \cite{bandini_2019}, the analogous problem is studied in a Markovian setup and it is shown that the value function is then the unique viscosity solution of an associated HJB equation formulated in terms of the so-called Lions derivative.
%; the article notably imposes no assumptions on the existence of a density. 

Finally, in \cite{bayraktar_2025}, a closely related problem where one controls a signal while observing only one of its driving noise processes is studied. In their setup, however, the volatility coefficient in front of the observed noise term does not depend on the signal itself and therefore their formulation does not encompass more general partially observable problems as in \cite{bandini_2018,bandini_2019}. However, their approach is more closely related to ours in the sense that they too establish comparison working directly on the space of probability measures (as opposed to some lifted space). Our results are effectively orthogonal though since the terms appearing in their controlled measure dynamics are complementary to those appearing in \eqref{eq:pi_intro}. %\cite{bayraktar_2025_arxiv}

\subsection*{Structure of the article}

The remainder of the article is organized as follows: In Section 2 we introduce our main problem of interest and express the dynamics of the involved processes in the observation filtration; in Section 3 we introduce our weak problem formulation as well as the piecewise constant approximate formulation and provide our main result; Section 4 presents our stability results and Section 5 is devoted to the proof of our main result.

\subsection*{Notation}

        Throughout, $\mu_0$ will be a given fixed probability measure on $\R$ with compact support. 
        We denote by $\Pc$ the space of probability measures on $\R$ whose support is contained in the support of $\mu_0$; we equip this space with the topology of weak convergence rendering it a Polish space. 
        Since $\mu_0$ has compact support, we can work with e.g. the Wasserstein-$1$ metric, which we denote by $\Wc$. 
        That $\mu_0$ has compact support also implies that the space $\Pc$ is compact.
        We denote by $\Pc^s$ the (closed) subset of probability measures supported in one single point.
        We write $C_b$ for the bounded continuous functions on $\R$. 
        Finally, for any $\mu\in\Pc$ and $\varphi:\R\to\R$, we set $\mu(\varphi):=\int_\R\varphi(x)\mu(\dx)$.

\section{Problem formulation}

\subsection{Controls and objective} \label{sec:formulation}
    Let $(\Omega,\Fc,\Fb,\P)$ be a given filtered probability space satisfying the usual conditions and supporting a Brownian motion $W$ and an independent $\F_0$-measurable random variable $X$ with law $\mu_0$; the latter will play the role of our unknown signal. We suppose throughout that $\mu_0$ has compact support; that is, there exists some $R>0$ such that 
    $$
    \supp(\mu_0)\subset(-R,R).
    $$

    In order to define our \emph{set of actions}, we let $K>0$ be arbitrary but fixed and introduce a closed set $\Uc$ such that
    \begin{align*}
        \Uc\subseteq\left\{v=(v_0,v_1,\dots):\textrm{$v_i\in\R$, $i\in\Zb_+$, and $\sum_{i=0}^\infty (R^iv_i)^2\le K$}\right\};
    \end{align*}
    we here consider the subspace topology inherited from the product topology and note that $\Uc$ thus defined is a compact Polish space.
    We let
    \begin{align}\label{eq:h}
        h(v,x):=\sum_{i=0}^\infty v_ix^i,
        \quad v\in\Uc, 
        \;x\in\supp(\mu_0);
    \end{align}
    we note that this function is well-defined and jointly measurable and that $x\mapsto h(v,x)$ is continuous and bounded on $\supp(\mu_0)$ uniformly in $v$ (\cf{} Lemma~\ref{lem:U} below).
    
    Our controls will be progressively measurable processes taking their values in this action space.
    More pertinently, we first define a set of pre-admissible controls by
    $$
        \Ac^{pre}:=\left\{\textrm{$u:\Omega\times[0,\infty)\to\Uc$: $u$ is $\Fb$-progressively measurable}\right\}.
    $$
    For $u\in\Ac^{pre}$, the observation process $Y^u=(Y^u_t)_{t\ge 0}$ is then given by
    \begin{align}\label{eq:y}
    \di Y^u_t=h(u_t,X)\dt+\di W_t,
    \quad Y^u_0=0.
    \end{align}
    For $u\in\Ac^{pre}$, we let $\mathcal{Y}^u=(\mathcal{Y}^u_t)_{t\ge 0}$ be the (control-dependent) observation filtration; that is, the filtration generated by the process $Y^u$ and completed by the $\P$-null sets.
    The set of \emph{admissible controls} is finally defined as
    $$
     \Ac:=\left\{u\in\Ac^{pre}:\textrm{$u$ is $\mathcal{Y}^u$-progressively measurable}\right\};
    $$
    for further discussion on how this type of closed-loop controls relate to alternative definitions, we refer to \cite[Remark~2.4]{cohen_2025}.

    We consider a discount rate $\beta>0$ and a measurable \emph{cost function} $k:\Pc\times\Uc\to\R$ to be given;
    throughout, we impose the following \emph{standing assumption}:
    \begin{assumption}\label{ass:U}
    The cost function $k$ is bounded and $k(\cdot,v)$ is continuous on $\Pc$ uniformly in $v\in\Uc$. 
    \end{assumption}
    For $u\in\Ac$, we define the posterior distribution $\pi^u=(\pi^u_t)_{t\ge 0}$ by
    \begin{align}\label{eq:filter}
    \pi^u_t=\mathrm{Law}(X|\mathcal{Y}^u_t);
    \end{align}
    note that $\pi^u$ takes values in $\Pc$ and that $\pi^u_0=\mu_0$.
    With each $u\in\Ac$, we then associate the following \emph{cost functional}:
    \begin{align}\label{eq:objective_filtering}
    J(u)=\E\int_0^\infty e^{-\beta t}k(\pi^u_t,u_t)\dt.
    \end{align}
    Our main problem of interest is the one of minimising $J(u)$ over all $u\in\Ac$.

    To further illustrate our setup, we give an example of a control problem which will fit into our setting, and for which our main results therefore are valid.
    
    \begin{example} \label{ex:sensor}
        Consider the following scenario: player A seeks to determine the (unknown and stationary) location of player B, denoted by $X$, without revealing his own position. The location of player B is known to lie within $I:=[-1, 1]$, i.e. the $X$ is compactly supported on $I$. Player A is equipped with a sensor characterised by  
        \begin{equation}\label{eq:example h}
            \tilde{h}(\tilde{v},x) = e^{-\lambda(x-\tilde{v})^2}, \quad 0<\lambda<\infty, \quad \tilde{v},x \in I, 
        \end{equation}
        where $\tilde{v}$ denotes the location where player A places his sensor. The strength of the signal he receives is then determined by the distance between his sensor and the true location $x$, with the strength decaying exponentially as the squared distance increases. 
        
        Consequently, the information available to player A evolves according to the stochastic process  
        $$
            \df Y^{\tilde{u}}_t = e^{- \lambda (X - \tilde{u}_t)^2} \df t + \df W_t, \quad Y^{\tilde{u}}_0 = 0,
        $$
        where $W$ is a standard Brownian motion independent of $X$, and where player A selects his position $\tilde{u}_t \in I$ at each time $t$ based on the observations collected up to that time. Player A begins with a prior distribution $\mu_0$ supported on $I$. The corresponding posterior distribution, or the filter, $\pi^{\tilde{u}}$ is defined as in \eqref{eq:filter}. Player A then considers an objective similar to \eqref{eq:objective_filtering} with a running cost $\tilde{k}: \mathcal{P} \times I \to \mathbb{R}$ defined by 
        \begin{equation}\label{eq:example cost function}
            \tilde{k}(\mu, \tilde{v})
            = 
            F \cdot \int_{[-1,1]} e^{-\theta (x-\tilde{v})^2} \mu(\mathrm{d}x)
            + 
            G \cdot  \mathbb{V}\text{ar}_{\mu}(\mathrm{id}), 
            \quad F,\ G \geq 0, 
        \end{equation}
        where $\mathbb{V}\text{ar}_{\mu}(\mathrm{id}):=\mu(\mathrm{id}^2) - \mu(\mathrm{id})^2$, the coefficient $F$ represents the fear associated with exposure, and $G$ the cost to player A of uncertainty in their knowledge of the true location of player B; we assume player B has the same type of sensor as player A, possibly with a different sensitivity parameter $\theta>0$. Hence, the function $\tilde{k}$ encodes the trade-off between fear and reward: the first term penalises being too close to player B, while the second term rewards a reduction in uncertainty about player B's location.  
        
        We anticipate a sensible expansion of the form $\tilde{h}(\tilde{v}, x) = \sum_{i=0}^{\infty} C_i(\tilde{v}) x^i$ allowing us to identify our set of controls as taking the form appropriate for the definitions above. In Appendix~\ref{sec:example_h}, we rigorously verify such an expansion and show that under this reformulation, the problem satisfies all the conditions imposed above. It will then follow from our main result that we can characterise the value of this problem as the unique continuous viscosity solution of a measure-valued HJB equation, and construct an $\varepsilon$-optimal piecewise constant control in feedback form.
    \end{example}

    \subsection{MVMs and SDEs}\label{sec:mvm_sde}
    
    For $u\in\Ac$, it holds that $\pi^u(f)$ is a martingale in the observation filtration, for any $f\in C_b$, and $\pi^u$ is therefore a so-called measure-valued martingale (MVM); see \cite[Definition~2.1]{cox_2017} or \cite{bayraktar_2018}, and also \eg{} \cite{dawson_1993} for an earlier account.
    More pertinently, taking the signal to be constant in the classical (controlled) Kushner-Stratonovich equation (see e.g. \cite[Proposition~7.2.8]{handel_2007}), we have that $\pi^u$ is the solution to\footnote{
    A-priori, setting $\Lc\equiv 0$ in \cite[Proposition~7.2.8]{handel_2007}, we obtain \eqref{eq:pi} (or, equivalently, \eqref{eq:ks}) for $\varphi\in C^2_b$ with all of its derivatives bounded. However, for any $\varphi\in C_b$, there exists a sequence $\{\varphi_n\}$ belonging to this class and converging pointwise to $\varphi$.
    Since, by Lemma~\ref{lem:U} below, $\{h(v,x):v\in\Uc,x\in\supp(\mu_0)\}$ is a bounded subset of $\R$,
    \begin{align*}
    \E\int_0^t\left(
    \int_\R
    (\varphi-\varphi_n)(x)\left(h(u_s,x)-\pi^u_s\left(h(u_s,\cdot)\right)\right)\pi^u_s(\dx)
    \right)^2\ds\xrightarrow[n\to\infty]{}0,
    \end{align*}
    which implies that \eqref{eq:pi} holds for each $\varphi\in C_b$.
    }
    \begin{align}\label{eq:ks}
    \di\pi^u_t(\varphi)
    =\big(\pi^u_t\big(\varphi(\cdot)h(u_t,\cdot)\big)-\pi^u_t\big(\varphi\big)\pi^u_t\big(h(u_t,\cdot)\big)\big)\big(\di Y^u_t
    -\pi^u_t(h(u_t,\cdot))\dt\big),
    \;\;\varphi\in C_b,
    \end{align}
    equipped with the initial condition $\pi^u_0=\mu_0$.
    Following the filtering literature, for $u\in\Ac$, we define the innovations process by 
    $$
    \di I^u_t=\left(h(u_t,X)-\pi^u_t(h(u_t,\cdot))\right)\dt+\di W_t,
    \quad I^u_0=0;
    $$ 
    it is a $(\mathcal{Y}^u,\P)$-Brownian motion. 
    Note that \eqref{eq:y} and \eqref{eq:ks} can then be re-written as
    \begin{align}\label{eq:y_innovation}
    \di Y^u_t=\pi^u_t(h(u_t,\cdot))\dt+\di I^u_t,
    \quad Y^u_0=0.
    \end{align}
    and
    \begin{align}\label{eq:pi}
    \di\pi^u_t(\varphi)
    =\big(\pi^u_t\big(\varphi(\cdot)h(u_t,\cdot)\big)-\pi^u_t\big(\varphi\big)\pi^u_t\big(h(u_t,\cdot)\big)\big)\di I^u_t,
    \;\varphi\in C_b,
    \;\;\pi^u_0=\mu_0;
    \end{align}
    the process $\pi^u$ can thus be viewed as a controlled MVM of the specific form considered in \cite{cox_2024} (\cf{} equation (2.3) therein).
    
    The links to MVMs and the above equations will become crucial for our upcoming analysis.
    Indeed, the benefit of linking our problem to equation \eqref{eq:pi} is that it enables relating the problem to a more standard (albeit infinite-dimensional) stochastic control problem featuring a controlled state-process being driven by a Brownian motion; the idea being that if one manages to construct a solution to this auxiliary (separated) problem, one should be able to construct a solution to the original problem too. 
    To see this, suppose for illustrative purposes that we have found a control in feedback form provided by a function $\hat u:\Pc\to\Uc$ for which the SDE
    \begin{align}\label{eq:sln_aux_new}
    \di\pi_t(\varphi)
    & =\left(\pi_t\left(\varphi(\cdot)h\left(\hat u(\pi_t),\cdot\right)\right)-\pi_t\left(\varphi\right)\pi_t\left(h\left(\hat u(\pi_t),\cdot\right)\right)\right)\di B_t,
    \;\varphi\in C_b,
    \;\;\pi_0=\mu_0,
    \end{align}
    admits a pathwise unique solution. Provided (i) that $\hat u$ is optimal for the problem of optimising the objective over solutions to \eqref{eq:sln_aux_new} for some suitable class of controls; (ii) that the value of this problem coincides with the value of our original problem of interest; and (iii) that by relying on uniqueness in law for \eqref{eq:sln_aux_new}, one may construct a control $u^*\in\Ac$ for which the objective attains the same value as when evaluated at the solution of this equation; then it would follow that this $u^*$ were optimal for the original problem. In the upcoming analysis, the role of $\hat u$ will be played by a piecewise constant control (in feedback form on a fine enough time grid) for which step (iii) simplifies considerably. In this way, we construct a piecewise constant and arbitrarily-close-to-optimal strategy for our main problem of interest. We will verify step (ii) by arguing that the value of the respective problems solve the same HJB equation; this will also give us further information about the value of the problem.

    One question is whether the reduction to piecewise constant controls (which is an approach we take from \cite{cohen_2025}) is necessary, and whether one could work directly with a control derived in feedback form. However it appears to be more difficult to verify that an arbitrary feedback control will be adapted to the filtration which arises from implementing the control, and hence that the control is admissible. This has been shown in much simpler settings however, see \eg{} \cite[Proposition~7.3.9]{handel_2007} for a result of this type in the much simplified Kalman-Bucy setting.

\begin{remark}[A finite-dimensional reduction]
    In \cite{cohen_2025}, the authors consider a setup closely related to ours: The dynamics of their observation process follows \eqref{eq:y} with $h(u_t,X)$ replaced by $b(t,Y^u_t,u_t)X$ for some function $b$ (they also let the signal be multi-dimensional and allow for a controlled volatility coefficient).
    For this setup, the unnormalised filter is given by $\sigma^u_t(\varphi)=\E^{\Qb^u_t}[\varphi(X)\Lambda^u_t|\mathcal{Y}^u_t]$, $\varphi\in C_b$, where $\Lambda^u_t=\di\P/\di\Q^u_t$ and $\Qb^u_t$ is the measure rendering $Y^u$ a Brownian motion on $[0,t]$; that is, $\Lambda^u_t=\mathcal{E}(X\Upsilon^u)_t$ with
    \begin{align}\label{eq:upsilon_new}
        \di\Upsilon^u_t=b(t,Y^u_t,u_t)\di Y^u_t,
        \quad\Upsilon^u_0=0.
    \end{align}
    Defining
    $ 
    F[\varphi](v,\gamma)=\int\varphi(x)e^{xv-\frac{1}{2}x^2\gamma}\mu_0(\dx),
    $
    it follows that
    \begin{align*}
        \sigma^u_t(\varphi)
        =\E^{\Qb^u_t}\big[\varphi(X)e^{X\Upsilon^u_t-\frac{1}{2}X^2\langle \Upsilon^u\rangle_t}|\mathcal{Y}^u_t\big]
        =F[\varphi]\left(\Upsilon^u_t,\langle\Upsilon^u\rangle_t\right).
    \end{align*}

    Note that, by It{\^o}'s formula, %and making use of the fact that $\frac{1}{2}F[\varphi]_{vv}+F[\varphi]_\gamma=0$ and $F[\varphi]_v=F[\varphi\cdot\mathrm{id}]$,
    \begin{align}\label{eq:dyn_rho_new}
        \di\sigma^u_t(\varphi)
        =F[\varphi]_v\left(\Upsilon^u_t,\langle\Upsilon^u\rangle_t\right)\di\Upsilon^u_t
        =\sigma^u_t(\varphi\cdot\mathrm{id})b(t,Y^u_t,u_t)\di Y^u_t,
    \end{align}
    which is what the Zakai equation reduces to in this setup.
    Moreover, in this setting the innovations process is given by 
    $$
    \di I^u_t=\left(X-\pi^u_t(\mathrm{id})\right)b(t,Y^u_t,u_t)\dt+\di W_t,
    \quad I^u_0=0,
    $$ 
    and in consequence, 
    \begin{align}\label{eq:y_sam_new}
    \di Y^u_t&=\frac{F[\mathrm{id}]}{F[1]}\left(\Upsilon^u_t,\langle\Upsilon^u\rangle_t\right)b(t,Y^u_t,u_t)\dt+\di I^u_t,
    \quad Y^u_0=0.
    %\di\Upsilon^u_t&=\frac{F[\mathrm{id}]}{F[1]}(\Upsilon^u_t,\langle\Upsilon^u\rangle_t)\frac{b^2}{\sigma^2}(t,Y^u_t,u_t)\dt+\frac{b}{\sigma}(t,Y^u_t,u_t)\di I^u_t,
    %\quad Y^u_0=0,\\
    %\di\langle\Upsilon^u\rangle_t&=\frac{b^2}{\sigma^2}(t,Y^u_t,u_t)\dt,
    %\quad Y^u_0=0.
    \end{align}
    Making use of \eqref{eq:dyn_rho_new}, \eqref{eq:y_sam_new} and applying again It{\^o}'s formula, it follows that
    \begin{align*}
        \di\pi^u_t(\varphi)
        =
        \left(\pi^u_t(\varphi\;\mathrm{id})-\pi^u_t(\varphi)\pi^u_t(\mathrm{id})\right)b(t,Y^u_t,u_t)\di I^u_t;
    \end{align*}
    for $b(t,Y^u_t,u_t)=b(u_t)$, this is precisely what \eqref{eq:pi} reduces to when $h(u_t,X)=b(u_t)X$.
 
    Crucially, the filter is here fully characterised by the current state of the processes $\Upsilon^u$ and $\langle\Upsilon^u\rangle$. 
    In \cite{cohen_2025}, this is carefully exploited: They express their objective --- the finite-horizon analogue of \eqref{eq:ext_y} with $\hat k(X,u_t)$ replaced by $\hat k(t,Y^u_t,X,u_t)$ --- as an expected value (under $\P$) of a functional depending on $(Y^u,\Upsilon^u,\langle\Upsilon^u\rangle)$; combining \eqref{eq:upsilon_new} with \eqref{eq:y_sam_new}, they then observe that the dynamics of $(Y^u,\Upsilon^u,\langle\Upsilon^u\rangle)$ can be described by a system of SDEs driven by a $\P$-Brownian motion.
    This leaves them with a \emph{finite-dimensional} separated problem, the study of which forms the basis for their approach.
    As mentioned in \cite[Remark~2.7]{cohen_2025}, a similar approach may be employed when restricting to controls which are polynomials of some fixed degree; we crucially go beyond such an assumption.
\end{remark}

\begin{remark}
    One might ask whether our results could be extended to include dependence on the observation process in the cost and observation functions.
    Including it in the cost $k$ means that the value function depends also on the factor $Y^u$; hence, one would need to verify that the viscosity and comparison theory we are citing from \cite{cox_2024} carries over to this case.
    %Indeed, a crucial component underlying our analysis is the It{\^o}-formula for functions of MVMs, say $f(\xi_t)$, developed in \cite{cox_2024}. Now, given a function $f:\R\times\Pc\to\R$, this formula gives the dynamics of the random field $f(y,\xi_t)$, $y\in\R$, and a straight-forward application of Ito-Wentzell's formula then yields the dynamics also of $f(Y_t,\xi_t)$.
    To include it in the observation function $h$, one would in addition need to establish existence of solutions to the pair of equations \eqref{eq:y_innovation}--\eqref{eq:pi} when $h$ also depends on the $Y$-component of that solution (with $u$ being constant and $I^u$ some fixed Brownian motion; \cf{} \eqref{eq:dynamics_controlled}) as well as study stability properties of that system of equations (\cf{} Theorem~\ref{thm:cont_MVM_poly}).
    We leave it for future research to investigate those interesting questions.
    %Here we have also focused on the infinite horizon case in order to avoid any explicit time-dependence in the value function; to cover this case one would in a similar way need to verify that the results in \cite{cox_2024} carries over. 
    %A related question of considerable interest is to investigate the effect of controlling also the diffusion term.
\end{remark}

\section{Weak and approximate formulation and main result}\label{section_5}

In this section we provide alternative problem formulations and give our main result, which on the one hand states that all those formulations coincide, and on the other hand characterises an arbitrarily-close-to-optimal strategy for our main problem of interest.

\subsection{Weak problem formulation}

We first introduce our weak formulation where there notably are no restrictions on the information used for the strategies:

\begin{definition}[Weak controls]
A tuple $(\Omega,\Fc,\Fb,\P,W,\xi,u)$ is a \emph{weak admissible control} if it satisfies:
\begin{enumerate}
    \item $(\Omega,\Fc,\Fb,\P)$ is a filtered probability space satisfying the usual conditions;
    \item $W$ is an $(\Fb,\P)$-Brownian motion;
    \item $u$ is an $\Uc$-valued $\Fb$-progressively measurable process;
    \item $\xi$ is a continuous $\Pc$-valued MVM such that 
    \begin{align}\label{eq:sde}
    \begin{split}
    %\di Y_s &=\xi_s(h(u_s,\cdot))\ds+\di W_s,\\
    \di\xi_s(\varphi)
     &=\big(\xi_s\big(\varphi h(u_s,\cdot)\big)-\xi_s(\varphi)\xi_s\big(h(u_s,\cdot)\big)\big)\di W_s,
     \quad\varphi\in C_b.
    \end{split}
    \end{align}
\end{enumerate}
\end{definition}

    To avoid overly cumbersome notation, we often call $(\xi,u)$ an admissible control without explicitly mentioning the other objects of the tuple.
    For $\mu\in\Pc$, we denote by $\mathcal{A}^{weak}(\mu)$ the set of all weak admissible controls $(\xi,u)$ which satisfy $\xi_0=\mu$; we note that this set is non-empty. 
    In turn, we define
    \begin{align}\label{eq:objective_weak}
    V^{weak}(\mu)=\inf_{\mathcal{A}^{weak}(\mu)}\E\int_0^\infty e^{-\beta t}k(\xi_t,u_t)\dt.
    \end{align}

    Our main result will establish that the value of this weak formulation agrees with the value of our original problem of interest.

    \subsection{Approximate problem formulation}\label{sec:piecewise_constant}

    We here introduce an approximation of the original problem of interest by restricting to piecewise constant controls.
    
    To this end, here, let $\Omega=C([0,\infty),\R)$, denote by $W$ the canonical process, and let $\P$ be the Wiener measure under which $W$ is a Brownian motion. 
    Let $\Fb^W$ denote the raw filtration generated by the canonical process and augmented by the null sets (but not completed).
    For $n\in\mathbb{N}$, define the dyadic step size $\delta_n=2^{-n}$ and consider the associated discrete time grid $\Tb^n=\{k\delta_n:k\in\mathbb{N}\}$.
    For $n\in\Nb$, let 
    \begin{align*}
        \mathcal{A}^n
        =\Big\{u:[0,\infty)\times\Omega\to\Uc :u_t=\sum_{r\in\Tb^n}u_r\boldsymbol{1}_{[r,r+\delta_n)}(t),\;\; t\in[0,\infty), 
        \; \textrm{$u_r$ is $\Fc^W_r$-meas.}\Big\}.
    \end{align*}
    For $\mu\in\Pc$ and $u\in\Ac^n$, $n\in\Nb$, let $\xi^{u;\mu}$ be the strong solution of
    \begin{align}\label{eq:dynamics_controlled}
    \di\xi_t(\varphi)
    &=\big(\xi_t\big(\varphi h(u_t,\cdot)\big)-\xi_t(\varphi)\xi_t\big(h(u_t,\cdot)\big)\big)\di W_t,
    \;\varphi\in C_b,
    \;\;\xi_0=\mu;
    \end{align}
    since $\mu$ has compact support and $h(v,\cdot)$ is continuous for $v\in\Uc$ (see Lemma~\ref{lem:U} below), existence and pathwise uniqueness of solutions to \eqref{eq:dynamics_controlled} is guaranteed by \cite[Theorem~2.1]{kunita_1971} (\cf{} also \cite[Theorems~V.5 and V.6]{szpirglas_1978}).
    For $\mu\in\Pc$ and $u\in\Ac^n$, $n\in\Nb$, we define
    \begin{align*}
    \hat J(u;\mu)
    :=\E
    \int_0^\infty e^{-\beta t}k\left(\xi^{u;\mu}_t,u_t\right)\dt
    =\E
    \sum_{r\in\Tb^n}\int_r^{r+\delta_n} e^{-\beta t}k\left(\xi^{u;\mu}_t,u_r\right)\dt.
    \end{align*}
    For $n\in\Nb$ and $\mu\in\Pc$, the approximate problem formulation is then given by
    \begin{align*}
        V^n(\mu)=\inf_{u\in\mathcal{A}^n}\hat J(u;\mu).
    \end{align*}

    Our main result will establish that these approximate problems converge, as the partition grows finer, to the value of our original problem of study.
    In order to formalise this, we define the limiting function, for $\mu\in\Pc$, by
    \begin{align*}
        V^+(\mu)
        =
        \lim_{n\to\infty}
        V^n(\mu);
    \end{align*}
    it is well defined since $V^n$ is bounded and monotone in $n$.

\subsection{Associated HJB equation}

    In order to formulate the HJB equation, we define, for $\mu\in\Pc$, $r\in\R$ and $\varphi\in C(\R^2)$,
    \begin{align*}
        H(\mu,r,\varphi)=\beta r+\sup_{v\in\Uc }\left\{-k(\mu,v)-\frac{1}{2}\int_{\R\times\R}\varphi(x,z)
        \sigma(v,\mu;\dx)\sigma(v,\mu;\dz)
        \right\},
    \end{align*}
    where, for $\mu\in\Pc$ and $v\in\Uc$,  
    \begin{align}\label{eq:def_sigma}
    \sigma(v,\mu;\dx)=(h(v,x)-\mu(h(v,\cdot)))\mu(\dx).
    \end{align}
    The HJB equation associated with our problem is then given by (\cf{} the dynamics in \eqref{eq:sde} and objective in \eqref{eq:objective_weak}):  
    \begin{align}\label{eq:hjb}
    H\left(\mu,u(\mu),\piann{u}{\mu}(\cdot,\cdot;\mu)\right)=0,
    \quad \mu\in\Pc;
    \end{align}
    for functions $u:\Pc\to\R$ we here make use of the notion of derivative as it was defined in \cite[Section~4]{cox_2024}, which is essentially the same notion as the one referred to as the \emph{linear functional derivative} in \cite[Section~5.4]{carmona_delarue_I}. 

    \begin{remark}\label{rem:boundary}
When $\mu = \delta_x \in \mathcal{P}^s$, \eqref{eq:hjb} simplifies to
\[
u(\delta_x)=\frac{1}{\beta}
\inf_{v\in\Uc}k(\delta_x,v).
\]
This can be interpreted as a kind of boundary condition: %(see also \cite[Remark~6.1]{cox_2024}): 
Since an MVM starting at a Dirac measure $\delta_x$ stays there for all times, by definition, 
$$
V^{weak}(\delta_x)=V^+(\delta_x)=\frac{1}{\beta}\inf_{v\in\Uc}k(\delta_x,v).
$$ 
\end{remark}

\subsection{Our main result}

    We are now ready to provide our main result; its proof is reported in Section~\ref{sec:proof_main}.
    When talking about solutions of equation \eqref{eq:hjb} we make use of the notion of viscosity solution as it was introduced in \cite{cox_2024}; for the reader's convenience, this definition is recalled in Definition~\ref{def:viscosity} below.
    We also recall the standing Assumption~\ref{ass:U}.
    
    \begin{theorem}\label{thm:main}
    %Let Assumption~\ref{ass:U} hold. 
    It holds that 
    \begin{align*} 
        V^{weak}(\mu_0)
        =\inf_{u\in\Ac}J(u)
        =V^+(\mu_0).
    \end{align*}
    
    Moreover, on $\Pc$,
    $$
    V^{weak}=V^+,
    $$
    and this function is the unique continuous viscosity solution of \eqref{eq:hjb}.
    %among functions $u:\Pc\to\R$ such that, for any $x_1,\dots,x_N\in\supp(\mu_0)$, the function $\tilde u$ given by \eqref{eq:reduction} is continuous on $\Delta^{N-1}$, and such that, for any $\mu\in\Pc$, there exists a sequence of finitely supported $\mu_n\in\Pc$ with $u(\mu)=\lim_{n\to\infty}u(\mu_n)$. 

    Finally, for any $\varepsilon>0$, we can find $n\in\Nb$ and a function $\hat u^n:\Pc\to\Uc$, such that $u^*\in\mathcal{A}$ given in feedback form by
    \begin{align}\label{eq:closing_constant}
        u^{*}_t=\sum_{r\in\Tb^n}\hat u^n\big(\pi^{u^*}_r\big)\boldsymbol{1}_{[r,r+\delta_n)}(t),
        \quad t\in[0,\infty),
    \end{align}
    satisfies 
    $$
    J(u^*)\le \inf_{u\in\Ac}J(u)+\varepsilon.
    $$
    \end{theorem}

    \begin{remark}
        The functions $V^+$ and $V^{weak}$ effectively correspond to a piecewise constant approximation and a weak formulation of the \emph{separated problem} associated with our main problem of study. Thanks to the above result, it is however possible to draw conclusions also about alternative formulations of the original problem itself: 
        (i) Restricting in our main problem of study to controls which are piecewise constant (over arbitrary partitions) would not alter the value of the problem. 
        (ii) Since the specifics of the underlying probability space do not have an impact on the value of the problem, considering a weak formulation where one also optimises over the underlying probability tuples would still yield the same value. 
        (iii) In the literature on stochastic control under partial information, it is common to consider so-called wide-sense controls which are adapted to a filtration which is larger than the observation filtration but still small enough for the filter to satisfy the associated filtering equation; allowing for controls of this type (adopting the definitions to the present context) would also not alter the value of the problem.
    \end{remark}

%\cred{
%    \begin{remark}[POSSIBLE EXTENSION]\label{rem:extension}
%    Let $\overline\Uc$ be the closure of $\Uc$ (w.r.t. the uniform norm, identifying $v$ with $h(v,\cdot)$). Then, imposing the assumption that $v\mapsto k(\mu,v)$ is continuous, $H(\mu,r,\varphi)$ will be the same no matter whether the supremum is taken over $\Uc$ or $\overline\Uc$; that is, the HJBs corresponding to $\Uc$ and $\overline\Uc$ are the same. In consequence, as long as one could manage to verify that Assumption (ii) of Theorem~\ref{thm:comparison} (\cf{} Lemma~\ref{lem:cont_V_weak} (ii)) is satisfied for $v=V^{weak}$ defined w.r.t. $\overline\Uc$, our entire result would a-posteriori hold also for $\overline\Uc$-valued controls! Indeed, this is the only bit that needs to be verified (it's fine to keep working with $\Uc$-valued processes for the auxiliary problem). To obtain this, one needs to prove the following: for any initial condition $\mu$ and control $u$, there exists a sequence of finitely supported $\mu_n$ and admissible controls $u_n$ such that the objective evaluated at $(\mu_n,u_n)$ converges to the objective evaluated at $(\mu,u)$ (our previous `approximation conjecture' stating that $\xi^{T^*_n\mu,u}_t$ and $\xi^{\mu,u}_t$ are `asymptotically equal in expectation' should be seen in light of this). %Conjecture~\ref{thm:cont_MVM_approx_m}
%    \end{remark}
%}

    \section{Stability properties of controlled MVMs}

    In this section we prove stability properties of our controlled MVMs with respect to the initial condition; these properties will become crucial in the upcoming analysis. Before proceeding to the key results, we verify some basic properties of our action space:
   
    \begin{lemma}\label{lem:U}
        The set $\{h(v,x):v\in\Uc, x\in\supp(\mu_0)\}$ is a bounded subset of $\R$ and the function $h(v,\cdot)$ is continuous on $\supp(\mu_0)$ uniformly in $v\in\Uc$.
    \end{lemma}
    
    \begin{proof}
        Let $0<\kappa<R$ such that $\supp(\mu_0)\subseteq[-\kappa,\kappa]$.
        Further, let $v\in\Uc$ and $x\in\supp(\mu_0)$. Then, for all $n,m\in\Nb$,  
        \begin{align}\label{eq:cauchy_conv}
            \sum_{i=n}^m|v_ix^i|
            \le 
            \sqrt{K}\sqrt{\sum_{i=n}^m\left(\frac{\kappa}{R}\right)^{2i}}.
        \end{align}
        Since $\sum_{i=0}^\infty\left(\kappa/R\right)^{2i}$ converges, we obtain that $\sum_{i=0}^\infty v_ix^i$ is absolutely convergent.
        A similar calculation yields the boundedness claim.
        Next, relying once again on \eqref{eq:cauchy_conv}, we have that for every $\varepsilon>0$, there exists some $n\in\Nb$, such that for all $v\in\Uc$ and $x,y\in\supp(\mu_0)$,
        \begin{align*}
            \left|h(v,y)-h(v,x)\right|
            %\le \left|\sum_{i=1}^n v_ix^i-\sum_{i=1}^n v_iy^i\right|+\varepsilon
            \le \sum_{i=1}^n|v_i||x^i-y^i|+\varepsilon
            \le \sqrt{K}\sqrt{\sum_{i=1}^n\left(\frac{x^i-y^i}{R^i}\right)^2}+\varepsilon.
        \end{align*}
        For each $i=1,\dots,n$, there exists some $\delta_i>0$, such that if $|x-y|\le \delta_i$, then $\frac{|x^i-y^i|}{R^i}\le \frac{\varepsilon}{\sqrt{Kn}}$. Provided that $|x-y|\le \min_{i=1,\dots,n}\delta_i$, it then holds that $\left|h(v,y)-h(v,x)\right|\le 2\varepsilon$, for all $v\in\Uc$, which completes the proof.
    \end{proof}

Our first stability result concerns finitely supported MVMs; its proof relies on stability properties for classical SDEs.

\begin{proposition}\label{thm:cont_MVM_finite}
    Let $N\in\Nb$ and consider some fixed $N$ points $\{x_1,\dots,x_N\}\subset\supp(\mu_0)$; let $\Pc_N=\{\mu\in\Pc:\supp(\mu)\subseteq\{x_1,\dots,x_N\}\}$. 
    \begin{enumerate}
        \item On any filtered probability space supporting a Brownian motion $W$ and a progressively measurable $\Uc$-valued process $u$, equation \eqref{eq:sde} equipped with an initial condition $\mu\in\Pc_N$ admits a pathwise unique solution; it is $\Pc_N$-valued \as{}
        \item For any $t>0$, there exists a constant $c$ depending on $N$ and $t$, such that for any filtered probability space supporting a Brownian motion $W$ and a progressively measurable $\Uc$-valued processes $u$, and for any initial conditions $\mu,\nu\in\Pc_N$,
    \begin{align*}
        \frac{\E\|((\xi^{u;\mu}_t-\xi^{u;\nu}_t)(\{x_1\}),\dots,(\xi^{u;\mu}_t-\xi^{u;\nu}_t)(\{x_N\}))\|}{\left\|((\mu-\nu)(\{x_1\}),\dots,(\mu-\nu)(\{x_N\}))\right\|}
        \le 
        c,
    \end{align*}
    where $\xi^{u;\mu}$ and $\xi^{u;\nu}$ denote the respective solutions to equation \eqref{eq:sde} and $\|\cdot\|$ is the Euclidean norm. 
    \end{enumerate}
\end{proposition}

\begin{proof}
    $(i)$. Suppose $\xi^{u; \mu}$ is a solution to (\ref{eq:sde}) with respect to an initial condition $\mu \in \mathcal{P}_N$. Then $\xi^{u; \mu}_t$ remains supported on $\{x_1, \ldots, x_N\}$ a.s. for all $t \ge 0$, and $\xi^{u; \mu}(\{x_i\})$, $i=1, \ldots, N$,  are non-negative martingales which satisfy
    \begin{align*}
        \mathrm{d} \xi^{u; \mu}_t(\{x_i\}) 
        &=
        \xi^{u; \mu}_t(\{x_i\}) \left( h(u_t, x_i) - \sum_{n=1}^N \xi^{u; \mu}_t(\{x_n\}) h(u_t, x_n) \right) \mathrm{d} W_t \label{eq: finite support MVM for xi^u,i}. 
    \end{align*}
    Let $\Delta^{N-1}$ be the $N$-dimensional standard simplex and define the function $\tilde{\sigma}: \mathcal{U} \times \Delta^{N-1} \to \mathbb{R}^N$ by
    \[
        \tilde{\sigma}(v, \theta)
        = 
        \left(
            \theta_1 \left( h(v, x_1) - \sum_{n=1}^N \theta_n h(v, x_n) \right),
            \ldots, 
            \theta_N \left( h(v, x_N) - \sum_{n=1}^N \theta_n h(v, x_n) \right)
        \right),    
    \] 
    If we denote $\tilde{\xi}^{u; \mu}_t = \left(\xi^{u; \mu}_t(\{x_i\}), \ldots, \xi^{u; \mu}_t(\{x_N\})\right)$, then $\tilde{\xi}^{u; \mu}$ is a $\Delta^{N-1}$-valued process which  satisfies the $N$-dimensional SDE
    \begin{equation}\label{eq: finite support MVM}
        \mathrm{d} \tilde{\xi}^{u; \mu}_t 
        = 
        \tilde{\sigma}(u_t, \tilde{\xi}^{u; \mu}_t) 
        \mathrm{d} W_t.
    \end{equation}
    Conversely, suppose $\tilde{\xi}^{u; \mu}$ is a solution to (\ref{eq: finite support MVM}) with initial condition $\tilde{\xi}^{u; \mu}_0 \in \Delta^{N-1}$. We define the measure-valued process $\hat{\xi}^{u; \mu}$ by $\hat{\xi}^{u; \mu}_t(f) \coloneqq \sum_{i=1}^N (\tilde{\xi}^{u; \mu}_t)_i f(x_i)$, $f \in C_b$, $t \ge 0$. Then, 
    \begin{align*}
        \mathrm{d} \hat{\xi}^{u; \mu}_t(f) 
        & = 
        \left( \sum_{i=1}^N (\tilde{\xi}^{u; \mu}_t)_i f(x_i) h(u_t, x_i) - \sum_{i=1}^N (\tilde{\xi}^{u; \mu}_t)_i f(x_i) \sum_{n=1}^N (\tilde{\xi}^{u; \mu}_t)_n h(u_t, x_n) \right) \mathrm{d}W_t
        \\
        & = 
        \left( \hat{\xi}^{u; \mu}_t \left( f h (u_t, \cdot) \right) - \hat{\xi}^{u; \mu}_t \left( f \right) \hat{\xi}^{u; \mu}_t \left( h (u_t, \cdot) \right) \right) \mathrm{d}W_t,
    \end{align*}
    which means that the process $\hat{\xi}^{u; \mu}$ is a solution to (\ref{eq:sde}). Since $\tilde{\xi}^{u; \mu}_0 \in \Delta^{N-1}$, $\hat{\xi}^{u; \mu}_0 \in \mathcal{P}_N$; since every measure-valued process satisfying (\ref{eq:sde}) has preserved mass, $\hat{\xi}^{u; \mu}$ is almost surely $\mathcal{P}_N$-valued. Therefore, the finite-dimensional SDE (\ref{eq: finite support MVM}) provides an equivalent representation of (\ref{eq:sde}) when the initial condition has finite support. It is thus sufficient to show that $\tilde{\sigma}(v, \theta)$ is Lipschitz continuous in $\theta$ uniformly for all $v \in \mathcal{U}$. 
    
    Let $\theta$ and $\eta$ be two distinct points in $\Delta^{N-1}$ and $\|h\| = \sup \{h(v, x): v \in \mathcal{U}, x \in \mathrm{supp}(\mu_0)\}$, which is well-defined by Lemma \ref{lem:U}. For each $k = 1, \ldots, N$, we have  
    \begin{align*}
        \left| \tilde{\sigma}_k(v, \theta) - \tilde{\sigma}_k(v, \eta) \right|
        & = 
        \left| (\theta_k - \eta_k) h(v, x_k) + \eta_k \sum_{n=1}^N \eta_n h(v, x_n) - \theta_k \sum_{n=1}^N \theta_n h(v, x_n) \right| 
        \\
        & \leq 
        | \theta_k - \eta_k| \|h\| + \sum_{n=1}^N |\eta_k \eta_n - \theta_k \eta_n + \theta_k \eta_n - \theta_k \theta_n | \|h\|
        \\
        & \leq 
        2 |\theta_k - \eta_k| \|h\| + \theta_k \sum_{n=1}^N  |\theta_n - \eta_n | \|h\|.
    \end{align*}
    Therefore, applying the triangle inequality and the inequality $(a_1 + \cdots + a_N)^2 \leq N (a_1^2 + \cdots + a_N^2)$, for all $a_i \in \mathbb{R}$, we have 
    \begin{align*}
        \left\| \tilde{\sigma}(v, \theta) - \tilde{\sigma}(v, \eta) \right\|
        \leq   
        2 \|h\| \|\theta - \eta\| + \|h\| \sqrt{N} \|\theta - \eta\|
        \leq 
        C \|\theta - \eta\|,
    \end{align*}
    where $C$ depends on $N$ and $\|h\|$. 
    
    \hfill 
 
    \noindent 
    $(ii)$. Denote $\tilde{\mu} = \left( \mu(\{x_1\}), \ldots, \mu(\{x_N\}) \right)$ and $\tilde{\nu} = \left( \nu(\{x_1\}), \ldots, \nu(\{x_N\}) \right)$. Let $\tilde{\xi}^{u; \mu}$ and $\tilde{\xi}^{u; \nu}$ be solutions to (\ref{eq: finite support MVM}) with initial conditions $\tilde{\mu}$ and $\tilde{\nu}$ respectively. Since the function $\tilde{\sigma}(v, \cdot)$ is uniformly Lipschitz for $v \in \mathcal{U}$, by \cite[Corollary 2.5.5]{krylov_2009}, 
    \begin{align*} 
        \left( 
            \mathbb{E} 
            \left[ 
                \left\| \tilde{\xi}^{u; \mu}_t - \tilde{\xi}^{u; \nu}_t \right\| 
            \right]
        \right)^2
        & \leq 
        4 \|\tilde{\mu} - \tilde{\nu}\|^2
        + 
        L(C) \| \tilde{\mu} - \tilde{\nu} \|^2 \int_0^t e^{(4C^2+1) (t-s)} \mathrm{d}s,  
    \end{align*}
    where $L(C)$ is a constant dependent on the uniform Lipschitz constant $C$ of $\tilde{\sigma}$. Hence,
    \[
        \mathbb{E} 
            \left[ 
                \left\| \tilde{\xi}^{u; \mu}_t - \tilde{\xi}^{u; \nu}_t \right\| 
            \right]
        \leq 
        c \|\tilde{\mu} - \tilde{\nu}\|, 
    \]
    where $c$ depends on $N$, $t$ and $\|h\|$. 
\end{proof}

Our second stability result concerns MVMs with general initial conditions. Its proof makes use of ideas similar to those employed in the proof of \cite[Theorem~2.1]{kunita_1971} %\cf{} \cite[Theorem~2.1.(ii)]{kunita_1971}
(\cf{} also the proof of \cite[Theorem~3]{fleming_1980}). The convergence is notably uniform in the control; this is needed for the upcoming proofs of Proposition~\ref{prop:properties_infinite} and the continuity of $V^+$ (\cf{} \eqref{eq:cont_V_+}).

\begin{theorem}\label{thm:cont_MVM_poly}
    Let $\mu\in\Pc$ and let $\mu_k$ be a sequence in $\Pc$ such that $\Wc(\mu_k,\mu)\to 0$ as $k\to\infty$. 
    Consider a filtered probability space supporting a Brownian motion $W$ and a sequence of progressively measurable $\Uc$-valued processes $u_k$, $k\in\Nb$. Suppose that $\xi^{u_k;\mu_k}$ and $\xi^{u_k;\mu}$, $k\in\Nb$, are solutions to equation \eqref{eq:sde} for $u=u_k$ and the initial conditions $\mu_k$ and $\mu$, respectively.
    Then, the solutions are unique (up to indistinguishability) and for any $t>0$,
    \begin{align*}
        \E\left|\xi^{u_k;\mu_k}_t(f)-\xi^{u_k;\mu}_t(f)\right|
        \xrightarrow[k\to\infty]{}0,
        \quad\textrm{for all $f\in C_b$}.
    \end{align*} 
    In particular, there exists a subsequence, which we still index by $k$, along which 
    \begin{align*}
        \Wc\left(\xi^{u_k;\mu_k}_t,\xi^{u_k;\mu}_t\right)
        \xrightarrow[k\to\infty]{}0,
        \quad\as{}
    \end{align*}
\end{theorem}

\begin{proof}
    Let $u$ be a $\mathcal{U}$-valued progressively measurable process such that (\ref{eq:sde}) admits solutions $\xi^{u; \mu_k}$ and $\xi^{u; \mu}$ with initial conditions $\mu_k$ and $\mu$, $k \in \mathbb{N}$. For ease of notation, we define    
    \[
        \delta^k_t(f) = \mathbb{E} \left[ \left|(\xi^{u; \mu_k}_t - \xi^{u; \mu}_t)(f)\right|^2 \right], 
        \quad 
        f \in C_b. 
    \]
    Since \eqref{eq:sde} is indifferent with respect to parallel shifts in $h(v,\cdot)$, without loss of generality, we may assume that $v_0\equiv 0$ for all $v\in\Uc$. 
    
    \underline{Step 1}: We first show that $\delta^k_t(f) \to 0$ for all $f \in C_b$ when $k \to \infty$, which implies the $L_1$-convergence. Denote by $\|f\|$ the sup-norm of $f \in C_b$ on $\mathrm{supp}(\mu_0)$, and $\|h\| = \sup \{h(v, x): v \in \mathcal{U}, x \in \mathrm{supp}(\mu_0)\}$, which is well-defined by Lemma \ref{lem:U}. Note that for any $k \in \mathbb{N}$, 
    \begin{equation}\label{eq: brutal bound}
        \delta^k_t(f)
        \leq
        4 \|f\|^2. 
    \end{equation}
    By (\ref{eq:sde}), and by adding and subtracting the hybrid term $\xi^{u; \mu_k}_s(f) \xi^{u; \mu}_s (h(u_s, \cdot))$, $\delta^k_t(f)$ can be rewritten as
    \begin{align*}
        \delta^k_t(f) 
        &=
        \mathbb{E} 
        \bigg[ 
            \Big| 
                (\mu_k - \mu)(f) 
                + 
                \int_0^t \Big\{ \left( \xi^{u; \mu_k}_s - \xi^{u; \mu}_s \right) \left( f h(u_s, \cdot) \right)
                \\
                & \quad \quad \quad \quad \quad \quad \quad \quad
                + \xi^{u; \mu_k}_s(f) \left( \xi^{u; \mu}_s \left( h \left(u_s, \cdot \right) \right)
                - \xi^{u; \mu_k}_s \left( h \left(u_s, \cdot \right) \right) \right) 
                \\
                & \quad \quad \quad \quad \quad \quad \quad \quad
                + \xi^{u; \mu}_s (h(u_s, \cdot)) \left( \xi^{u; \mu}_s(f) - \xi^{u; \mu_k}_s(f) \right)  
                \Big\} \mathrm{d}W_s
            \Big|^2 
        \bigg]. 
    \end{align*}
    Applying the It\^o isometry and the inequality $(a_1 + \cdots + a_n)^2 \leq n(a_1^2 + \cdots + a_n^2)$, for all $a_i \in \mathbb{R}$, with $n=2$ and $n=3$, we get 
    \begin{align}\label{eq: stability polynomial R case iteration raw form}
        \delta^k_t(f)
        & \leq
        2 |(\mu_k - \mu)(f)|^2
        +
        6 \int_0^t \mathbb{E} 
        \left[ 
            \left| 
                \left( \xi^{u; \mu_k}_s - \xi^{u; \mu}_s \right)
                \left( f h(u_s, \cdot) \right)
            \right|^2 
        \right]
        \mathrm{d}s
        \nonumber
        \\
        & \quad   
        + 
        6 \|f\|^2 \int_0^t \mathbb{E}
        \left[
            \left|
                \left( \xi^{u; \mu_k}_s - \xi^{u; \mu}_s \right)
                \left( h(u_s, \cdot) \right)
            \right|^2
        \right] 
        \mathrm{d}s
        + 
        6 \|h\|^2 \int_0^t \delta^k_s(f) \mathrm{d}s. 
    \end{align}
    Note that for all $f \in C_b$, 
    \begin{align*}\label{eq: stability polynomial R case mid term one}
        \mathbb{E} 
        \left[
                \left( \xi^{u; \mu_k}_s - \xi^{u; \mu}_s \right)
                \left( f h(u_s, \cdot) \right)^2 
        \right] 
        & =
        \mathbb{E} 
        \left[
                \left( \xi^{u; \mu_k}_s - \xi^{u; \mu}_s \right)
                \left( f \sum_{i=1}^{\infty} (u_s)_i\mathrm{id}^i \right)^2
        \right]
        \\
        & =
        \mathbb{E} 
        \left[
            \left(
                \sum_{i=1}^{\infty} 
                (u_s)_i R^i 
                \left( \xi^{u; \mu_k}_s - \xi^{u; \mu}_s \right)
                \left( f \frac{\mathrm{id}^i}{R^i} \right) 
            \right)^2
        \right]
        \\
        &\leq
        \mathbb{E}
        \left[
            \left(
                \sum_{i=1}^{\infty}
                (u_s)_i^2 R^{2i}
            \right)
            \left(
                \sum_{i=1}^{\infty}
                    \left( \xi^{u; \mu_k}_s - \xi^{u; \mu}_s \right)
                    \left( f \frac{\mathrm{id}^i}{R^i} \right)^2
            \right)
        \right]
        \\
        &\leq 
        K \sum_{i=1}^{\infty} \delta_s^k \left( f \frac{\mathrm{id}^i}{R^{i}} \right). 
    \end{align*}
    Substituting the above estimate in (\ref{eq: stability polynomial R case iteration raw form}), we obtain the estimate 
    \begin{align*}
        \delta_t^k(f) 
        &\leq 
        2 |(\mu_k - \mu)(f)|^2
        \\
        & \quad 
        + 6 \int_0^t 
        \left(
            K \sum_{i=1}^{\infty} \delta_s^k \left( f \frac{\mathrm{id}^i}{R^{i}} \right) 
            +
            K \|f\|^2 \sum_{i=1}^{\infty} \delta_s^k \left( \frac{\mathrm{id}^i}{R^{i}} \right)
            +
            \|h\|^2 \delta_s^k(f)
        \right)
        \mathrm{d}s.
    \end{align*}    
    To further ease the notation, we denote 
    \[
        D_k(f) = 2|(\mu_k - \mu)(f)|^2, 
        \quad
        \psi = \frac{\mathrm{id}}{R}, 
        \quad
        L = \sum_{i=1}^{\infty} \left\| \psi^i \right\|^{2}, 
        \quad
        M = 6 \max\left(K, \|h\|^2\right), 
    \]
     which allows us to rewrite the above estimate as %and get a neater estimate 
    \begin{align} \label{eq: iteration form}
        \delta_t^k(f) 
        &\leq
        D_k(f) 
        +
        M \int_0^t 
        \left(
            \sum_{i=1}^{\infty} \delta_s^k \left( f \psi^i \right) 
            +
            \|f\|^2 \sum_{i=1}^{\infty} \delta_s^k \left( \psi^i \right)
            +
            \delta_s^k(f)
        \right)
        \mathrm{d}s.
    \end{align}
    We first apply the bound (\ref{eq: brutal bound}) to (\ref{eq: iteration form}) and get a new estimate 
    \begin{align}\label{eq: first iteration estimate}
        \delta_t^k(f) 
        & \leq 
        D_k(f)
        +
        M \int_0^t 
        \Bigg(
            \sum_{i=1}^{\infty} 4 \left\| f \psi^i \right\|^2
            +
            \|f\|^2 \sum_{i=1}^{\infty} 4 \left\| \psi^i \right\|^2
            +
            4 \|f\|^2
        \Bigg)
        \mathrm{d}s 
        \nonumber 
        \\
        & \leq 
        D_k(f)
        +  
        4M (2L + 1) \|f\|^2 t. 
    \end{align}
    We can apply (\ref{eq: first iteration estimate}) to $f \psi^{i}$, $\psi^{i}$, $i \in \mathbb{N}$, and $f$ in (\ref{eq: iteration form}), and again obtain a second new estimate 
    \begin{align}\label{eq: second iteration estimate}
        \delta_t^k(f)
        & \leq
        D_k(f)
        + 
        M \left( \sum_{i=1}^{\infty} D_k(f \psi^i) 
        +
        \|f\|^2 \sum_{i=1}^{\infty} D_k(\psi^i) 
        + 
        D_k(f) \right) t 
        \nonumber 
        \\
        & \quad \quad \quad \quad 
        + 4 M^2 (2L+1)^2 \frac{t^2}{2!}
        \|f\|^2. 
    \end{align}
    Again, we can apply (\ref{eq: second iteration estimate}) to $f \psi^{i}$, $\psi^{i}$, $i \in \mathbb{N}$, and $f$ in (\ref{eq: iteration form}), and obtain a third new estimate 
    \begin{align}
        \delta_t^k(f)
        &\leq
        D_k(f)
        +
        M 
        \left(
            \sum_{i=1}^{\infty} D_k(f \psi^i)
            + 
            \|f\|^2 \sum_{i=1}^{\infty} D_k(\psi^i)
            + 
            D_k(f)
        \right) t
        \nonumber 
        \\
        & \quad \quad \quad \quad 
        +
        M^2
        \bigg(
            \sum_{i_1, i_2 = 1}^{\infty} D_k(f \psi^{i_1 + i_2})
            +
            L \|f\|^2 \sum_{i=1}^{\infty} D_k(\psi^i)
            +
            \sum_{i=1}^{\infty} D_k(f \psi^i)
            \nonumber 
            \\
            & \;\; \quad \quad \quad \quad \quad %\quad 
            + 
            \|f\|^2 \sum_{i_1, i_2 = 1}^{\infty} D_k(\psi^{i_1 + i_2})
            +
            L \|f\|^2 \sum_{i=1}^{\infty} D_k(\psi^i) 
            + 
            \|f\|^2 \sum_{i=1}^{\infty} D_k(\psi^i) 
            \nonumber 
            \\
            & \;\; \quad \quad \quad \quad \quad %\quad 
            +
            \sum_{i=1}^{\infty} D_k(f \psi^i)
            +
            \|f\|^2 \sum_{i=1}^{\infty} D_k(\psi^i)
            +
            D_k(f)
        \bigg) \frac{t^2}{2!}
        \nonumber 
        \\
        & \quad \quad \quad \quad 
        +
        4 M^3 (2L+1)^3 \frac{t^3}{3!} \|f\|^2. 
        \nonumber 
    \end{align}
    We do this iteratively to (\ref{eq: iteration form}). If we regard (\ref{eq: first iteration estimate}) as the first iteration, then for the $m^{\text{th}}$ iteration, $m \geqslant 2$, $m \in \mathbb{N}$, we obtain the estimate 
    \begin{align}\label{eq: m th iteration}
        \delta^k_t(\textbf{}f)
        & \leq
        D_k(f)
        +
        M \Delta^k_1 t + \cdots + M^{m-1} \Delta^k_{m-1} \frac{t^{m-1}}{(m-1)!}
        +
        4 M^m (2L+1)^m \frac{t^m}{m!} \|f\|^2,
    \end{align}
    where each $\Delta^k_j, j=1, \ldots, m-1$, consists of
    of $3^j$ terms and is a linear combination of components of the forms  
    \begin{align}\label{eq: m iteration components}
        \sum_{i_1, \ldots, i_j = 1}^{\infty} D_k(f \psi^{i_1 + \dots + i_j}), 
        \ldots, 
        D_k(f), 
        \sum_{i_1, \ldots, i_j = 1}^{\infty} D_k(\psi^{i_1 + \dots + i_j}), 
        \ldots,  
        \sum_{i=1}^{\infty} D_k(\psi^i), 
    \end{align}
    with coefficients depending on $L$ and $\|f\|$. Note that for any $j = 1, \ldots, m-1$, 
    \begin{align*}
        \sum_{i_1, \ldots, i_j = 1}^{\infty} D_k(f \psi^{i_1 + \cdots + i_j}) 
        & \leq
        \sum_{i_1, \ldots, i_j = 1}^{\infty} 8 \|f \psi^{i_1 + \cdots + i_j}\|^2
        \leq
        8 \|f\|^2 L^j , 
    \end{align*}
    which means the series is convergent. Thus, for any $\varepsilon > 0$, there exists $N \in \mathbb{N}$ such that 
    \[
        \sum_{i_1, \ldots, i_j = 1}^{\infty} D_k(f \psi^{i_1 + \cdots + i_j})
        <
        \sum_{i_1, \ldots, i_j = 1}^{N} D_k(f \psi^{i_1 + \cdots + i_j}) 
        +
        \varepsilon. 
    \]
    Since $f\psi^i$ are Lipschitz for all $i \in \mathbb{N}$ and $\varepsilon$ is arbitrary, this means that every component in (\ref{eq: m iteration components}) converges to $0$ as $k \to \infty$; since each $\Delta_j^k$ has finitely many components with finite coefficients, it follows that all the $\Delta_j^k$, $j=1, \ldots, m-1$, converge to $0$ when $k \to \infty$. Therefore, from (\ref{eq: m th iteration}) we get 
    \begin{equation*}
        \lim_{k \to \infty} \delta^k_t(f) 
        \leq 
        4 M^m (2L+1)^m \frac{t^m}{m!} \|f\|^2. 
    \end{equation*} 
    Sending $m \to \infty$, we obtain the desired result. Note, in particular, that the estimates (\ref{eq: brutal bound}) and (\ref{eq: iteration form}) do not depend on $u$, so all the consequent iteration estimates do not depend on $u$. Hence, the convergence is uniform in $u$. 

    \hfill

    \noindent
    \underline{Step 2}: Since we have the $L_1$-convergence, for every $f \in C_b$, there is a subsequence, which we still index by $k$, such that 
    \[
        \left| \xi^{u_k; \mu_k}_t(f) - \xi^{u_k; \mu}_t(f) \right| \xrightarrow[k \to \infty]{} 0 \quad \text{a.s.} 
    \]
    We pick a countable subset $\{f_i\}_{i=1}^{\infty} \subset C_b$, where $f_1 \equiv 1$ and $\{f_i\}_{i=2}^{\infty}$ coincides with the polynomials with rational coefficients on $\mathrm{supp}(\mu_0)$. By a diagonalisation argument, there exists a subsequence, which we still index by $k$, such that 
    \[
        \left| \xi^{u_k; \mu_k}_t(f_i) - \xi^{u_k; \mu}_t(f_i) \right| \xrightarrow[k \to \infty]{} 0 \quad \text{a.s.}, \quad \text{for all } f_i \in \{f_i\}_{i=1}^{\infty}. 
    \]
    Since there are countably many $f_i$, there exists a probability one set, say $A$, such that for all $\omega \in A$, 
    \[
        \left| \xi^{u_k; \mu_k}_t(\omega) (f_i) - \xi^{u_k; \mu}_t(\omega) (f_i) \right| \xrightarrow[k \to \infty]{} 0, \quad \text{for all } f_i \in \{f_i\}_{i=1}^{\infty}. 
    \]
    Note that $\{f_i\}_{i=1}^{\infty}$ is an algebra which strongly separates points in $\text{supp}(\mu_0)$. Therefore, by \cite[Theorem 3.4.5(b)]{ethier_1986}, along this subsequence we have $\xi^{u_k; \mu_k} \Rightarrow \xi^{u_k; \mu}$ a.s. Since $\text{supp}(\mu_0)$ is compact, this gives us the desired result. 

    \hfill

    \noindent 
    \underline{Step 3}: We show pathwise uniqueness of solutions to \eqref{eq:sde}. Let $\xi^{u; \mu}$ and $\tilde\xi^{u; \mu}$ be two solutions to \eqref{eq:sde} defined on the same probability space and with respect to the same Brownian motion $W$, control $u$ and initial condition $\mu$. Define, for $f \in C_b$,
    \[
        \rho_t(f) = \mathbb{E} \left| \xi^{u; \mu}_t(f) - \tilde\xi^{u; \mu}_t(f) \right|^2. 
    \]
    Using similar arguments and notation as in step 1, we have
    $$
        \rho_t(f) \leq M \int_0^t \left( \sum_{i=1}^{\infty} \rho_s(f \psi^i) + \|f\|^2 \sum_{i=1}^{\infty} \rho_s(\psi^i) + \rho_s(f) \right) \df s.
    $$
    By applying the iteration procedure (note that it simplifies considerably since all $D_k$ vanish), we obtain that for any $m \in \mathbb{N}$,
    $$
        \rho_t(f) \leq 4 M^m (2L+1)^m \frac{t^m}{m!} \|f\|^2.
    $$
    Sending $m \to \infty$, we get $\rho_t(f) = 0$ for all $t\geq 0$ and $f \in C_b$. By arguments similar to those used in step 2, using the countable sequence $\{f_i\}_{i=1}^{\infty}$ and \cite[Theorem 3.4.5(a)]{ethier_1986}, we get $\xi^{u; \mu}_t = \tilde\xi^{u; \mu}_t$ a.s. for all $t \geq 0$. Moreover, by \cite[Remark 2.3]{cox_2024}, the trajectories of $\xi^{u; \mu}$ and $\tilde\xi^{u; \mu}$ are continuous. Hence, the two processes are indistinguishable, completing the proof. 
\end{proof}

\section{Viscosity theory and proof of main result}\label{sec:proof}

    In order to establish our main result, we will study in detail the limit of the auxiliary problem. 
    By definition, on $\Pc$,
    \begin{align}\label{eq:proof_main_ordering_continuity}
    V^{weak} \le V^+. 
    \end{align}
    Our strategy is to show that $V^+$ is a viscosity subsolution of an HJB equation, a property that ultimately relies on the fact that the DPP holds for the auxiliary problem. 
    From \cite{cox_2024}, we already know that $V^{weak}$ is a viscosity solution of the same equation, and that this equation does satisfy a comparison principle.
    Having established that the involved functions are sufficiently smooth, we may thus deduce that \eqref{eq:proof_main_ordering_continuity} holds with equality. 
    Constructing an optimal control for an approximate problem with sufficiently fine grid and closing the loop along the lines outlined in Section~\ref{sec:mvm_sde}, we may then complete the proof of our main result.
    The remainder of this section is devoted to carrying out this scheme.
    
    We note that a similar approach was used for the closely related finite-dimensional problem studied in \cite{cohen_2025}.
    More pertinently, the idea of justifying an approximation by relying on comparison for the associated HJB equation effectively goes back to \cite{barles_1991}, %\cite[Theorem~2.1]{barles_1991}
    meanwhile, the approach of sandwiching the main problem of interest between functions which more easily can be verified to satisfy the HJB equation, resembles the stochastic Perron method developed in \cite{bayraktar_2013}. The results in this section develop arguments of this type for the infinite-dimensional problem at hand.
    
    We first recall the definition of viscosity solutions which we make use of in this paper. 
    To this end, let $C^2(\Pc)$ denote the class of functions which are twice continuously differentiable in the sense of \cite[Definition~4.7]{cox_2024}. %\cite[Definitions~4.1~and~4.7]{cox_2024} 
    Further, recall that continuous MVMs have decreasing support in the sense that, with probability one, $\supp(\xi_t)\subseteq\supp(\xi_s)$, for $s\le t$; see \cite[Remark~2.3.(ii)]{cox_2024}. Motivated by this fact, consider the partial order $\preceq$ defined on $\Pc$ by
    \[
        \mu\preceq\nu \quad\Longleftrightarrow\quad \supp(\mu)\subseteq\supp(\nu);
    \]
    MVMs are then decreasing with respect to this order. Utilising this property, the notion of viscosity solutions is then defined as follows:
    
\begin{definition}[Definition~6.4 in \cite{cox_2024}]\label{def:viscosity}
    A function $u:\Pc\to\R$ is a \emph{viscosity subsolution} (resp. \emph{supersolution}) of \eqref{eq:hjb} if    
    \begin{align*}
        \liminf_{\mu\to\bar\mu,\mu\preceq\bar\mu}%{\tiny \begin{array}{c}\mu\to\bar\mu\\\mu\preceq\bar\mu\end{array}}
        H\bigg(\mu,\varphi(\mu),\piann{\varphi}{\mu}(\cdot,\cdot;\mu)\bigg)\le 0
        \;\; \bigg(\textrm{resp.}\limsup_{\mu\to\bar\mu,\mu\preceq\bar\mu}%{\tiny \begin{array}{c}\mu\to\bar\mu\\\mu\preceq\bar\mu\end{array}}
        H\bigg(\mu,\varphi(\mu),\piann{\varphi}{\mu}(\cdot,\cdot;\mu)\bigg)\ge 0\bigg)
    \end{align*}
    holds for all $\bar\mu\in\Pc$ and $\varphi\in C^2(\Pc)$ such that 
    \begin{align*}
        \varphi(\bar\mu)=\limsup_{\mu\to\bar\mu,\;\mu\preceq\bar\mu}%{\tiny \begin{array}{c}\mu\to\bar\mu\\\mu\preceq\bar\mu\end{array}}
        u(\mu)
        \quad
        \left(\textrm{resp. }\varphi(\bar\mu)=\liminf_{\mu\to\bar\mu,\;\mu\preceq\bar\mu}%{\tiny \begin{array}{c}\mu\to\bar\mu\\\mu\preceq\bar\mu\end{array}}
        u(\mu)\right), 
    \end{align*}
    and $\varphi(\mu)\ge u(\mu)$ (resp. $\varphi(\mu)\le u(\mu)$) for all $\mu\preceq\bar\mu$.
    
    It is a viscosity solution if it is both a sub- and supersolution. 
    \end{definition}

    \subsection{DPP for piecewise constant controls}
 
    The main result of this section is the dynamic programming principle for the auxiliary problem.
    We start by establishing a continuity property of the auxiliary value function.

\begin{proposition}\label{prop:properties_infinite}
    Let $n\in\Nb$.
    The function $\hat J(u;\cdot)$ is continuous on $\Pc$ uniformly in $u\in\Ac^n$. 
    In consequence, the function $V^n$ is continuous on $\Pc$. 
\end{proposition}

\begin{proof}
    The second part follows immediately from the first since, for any $\mu,\nu\in\Pc$,
    $$ 
        \left|V^{n}(\mu)-V^{n}(\nu)\right|
        \le
        \sup_{u\in\Ac^n}\big|\hat J(u;\mu)-\hat J(u;\nu)\big|.
    $$
    To argue the first part, let $\mu\in\Pc$ and consider a sequence $\mu_k$ in $\Pc$ such that $\mu_k\to\mu$, and an arbitrary sequence $u_k$ in $\Ac^n$.
    For any $\varepsilon>0$, there exists some $T>0$ such that, for all $k\in\Nb$,  
    \begin{align}\label{eq:proof_J_n_1_infinite}
        \big|\hat J(u_k;\mu_k)-\hat J(u_k;\mu)\big|
        &\le 
        \E\int_0^T\left|k\big(\xi^{u_k;\mu_k}_t,(u_k)_t\big)-k\big(\xi^{u_k;\mu}_t,(u_k)_t\big)\right|\dt
        +\varepsilon\nonumber\\
        &\le 
        Tw\left(\E\;\Wc\left(\xi^{u_k;\mu_k}_T,\xi^{u_k;\mu}_T\right)\right)+\varepsilon,
    \end{align}
    where we used that, by the standing Assumption~\ref{ass:U}, $\mu\mapsto k(\mu,v)$ is continuous, uniformly in $v\in\Uc$, and thus admits a concave and continuous modulus of continuity $w$ on the compact set $\Pc$, and the fact that $\Wc(\xi^{u_k;\mu_k},\xi^{u_k;\mu})$ is a submartingale for $k\in\Nb$ (\cf{} \cite[Proposition~4.1.(i)]{beiglbock_2017}).
    
    Now, for any subsequence of $(\mu_k,u_k)$, according to Theorem~\ref{thm:cont_MVM_poly}, there exists a further subsequence, which we still index by $k$, along which $\Wc(\xi^{u_k;\mu_k}_T,\xi^{u_k;\mu}_T)$ converges to zero \as{}, and along which the first term on the right-hand side of \eqref{eq:proof_J_n_1_infinite} thus converges to zero; since any subsequence of the original sequence admits a further subsequence with this property, the convergence must hold also along the original sequence. 
    Since $\varepsilon>0$ was arbitrarily chosen, this completes the proof. 
\end{proof}
 
    Thanks to the above continuity properties, we may establish the DPP following the arguments developed in \cite{bouchard_2011}. By use of the DPP, we may then prove existence of an $\varepsilon$-optimal strategy for the auxiliary problem. 
    We present those results next; 
    for completeness, the proofs are provided in Appendix~\ref{sec:dpp}.
    
    For $v\in\Uc$, we write $\xi^{v;\mu}$ for the solution to \eqref{eq:dynamics_controlled} using the constant control $u_s\equiv v$. 
    
    \begin{theorem}[DPP]\label{thm:dpp}
        Let $n\in\Nb$.
        For every $\mu\in\Pc$, 
        \begin{align*}
        V^n(\mu)
        =
        \inf_{v\in\Uc}
        \E\left[\int_0^{\delta_n}e^{-\beta t}k\left(\xi^{v;\mu}_t,v\right)\dt+e^{-\beta\delta_n}V^n\left(\xi^{v;\mu}_{\delta_n}\right)\right].
        \end{align*}
    \end{theorem}
    \begin{corollary}\label{cor:optimiser_existence}
        For any $n\in\Nb$ and $\varepsilon>0$, there exists a measurable function $\hat u^n:\Pc\to\Uc$, such that for $\mu\in\Pc$,
        \begin{align*} 
        \hat J(u^{*,n};\mu)\le V^n(\mu)+\varepsilon,
        \end{align*}
        where $u^{*,n}\in\Ac^n$ is given by
        \begin{align}\label{eq:optimal_u}
            u^{*,n}_t=\sum_{r\in\Tb^n}\hat u^n\left(\xi^{u^{*,n},\mu}_r\right)\boldsymbol{1}_{[r,r+\delta_n)}(t),
            \quad t\in[0,\infty).
        \end{align}
    \end{corollary}

    \subsection{Subsolution property of limiting problem}

    In this section we establish the subsolution property of $V^+$.
    To ease the notation, we introduce the effective state space for an MVM starting at a measure $\bar\mu\in\Pc$; it is given by the set
\begin{equation*}
D_{\bar\mu} = \{\mu\in\Pc\colon \mu\preceq\bar\mu\}.
\end{equation*}
    Making use of the fact that, for any $u\in C^2(\Pc)$, the left-hand side of \eqref{eq:hjb} is lower semicontinuous in $\mu$, and the fact that $V^+$ is upper semicontinuous on $\Pc$ (being the decreasing limit of continuous functions by Proposition~\ref{prop:properties_infinite}), we also note that $V^+$ is a subsolution of \eqref{eq:hjb}, if and only if, for any $\bar\mu\in\Pc$ and $\varphi\in C^2(\Pc)$, one has the implication
\begin{align}\label{eq:subsolution_alternative_continuity}
\text{$\varphi(\bar\mu)=V^+(\bar\mu)$ and $\varphi>V^+$ on $D_{\bar\mu}\setminus\{\bar\mu\}$} 
\;\;\Longrightarrow\;\;
%\liminf_{\mu\to\bar\mu,\;\mu\preceq\bar\mu}
H\left(\bar\mu,\varphi(\bar\mu),\piann{\varphi}{\mu}(\cdot,\cdot;\bar\mu)\right)\le 0.
\end{align} 
Indeed, the `only if' is immediate. Meanwhile, the `if' follows by use of the same arguments as employed in the proof of \cite[Lemma~6.6]{cox_2024}.

\begin{lemma}\label{lem:pick_mu_n_continuity}
    Let $\bar\mu\in\Pc$ and $\varphi\in C^2(\Pc)$ such that $\varphi(\bar\mu)=V^+(\bar\mu)$ and $V^+-\varphi<0$ on $D_{\bar\mu}\setminus\{\bar\mu\}$.
    Then, there exist sequences $m_n\to\infty$ and $\mu_n\to\bar\mu$, $\mu_n\in D_{\bar\mu}$, such that $\mu_n$ attains the maximum of $V^{m_n}-\varphi$ on $D_{\bar\mu}$ and $V^{m_n}(\mu_n)\to V^+(\bar\mu)$.
    
    In particular, defining $\kappa_n:=(V^{m_n}-\varphi)(\mu_n)\ge 0$
    then $\kappa_n\to 0$ and $V^{m_n}\le \varphi+\kappa_n$ on $D_{\bar\mu}$. 
\end{lemma}

\begin{proof}
    We first note that $D_{\bar\mu}$ is a closed subset of the compact set $\Pc$; hence, $D_{\bar\mu}$ is itself a compact set.
    For $n\in\Nb$, let $N_n\ge n$ and $\varepsilon_n\le 1/n$ such that, for all $m\ge N_n$, $\mu,\nu\in D_{\bar\mu}\cap B_{\varepsilon_n}(\bar\mu)$, 
    \begin{equation*}
        V^m(\mu)\le V^+(\bar\mu)+1/n
        \quad\textrm{and}\quad 
        \left|\varphi(\nu)-\varphi(\mu)\right|\le 1/n;
    \end{equation*}
    existence is guaranteed by the fact that $V^n$, $n\in\Nb$, is continuous and decreasing in $n$ (\cf{} Proposition~\ref{prop:properties_infinite}). By the upper semicontinuity of $V^+$ (\cf{} again Proposition~\ref{prop:properties_infinite}), there exists $\eta_n>0$ such that $V^+-\varphi\le-2\eta_n$ on $D_{\bar\mu}\setminus B_{\varepsilon_n}(\bar\mu)$. By possibly making $N_n$ even larger, we may ensure that, for all $m\ge N_n$, $\mu\in D_{\bar\mu}\setminus B_{\varepsilon_n}(\bar\mu)$,
    \begin{align}\label{eq:pick_sequence_continuity}
        (V^m-\varphi)(\mu)<-\eta_n.
    \end{align}
    Indeed, for any $\nu\in D_{\bar\mu}\setminus B_{\varepsilon_n}(\bar\mu)$, there exist $N^\nu$ and $\varepsilon^\nu$ such that \eqref{eq:pick_sequence_continuity} holds for all $m\ge N^\nu$, $\mu\in B_{\varepsilon^\nu}(\nu)$; a finite covering argument yields the claim. 

    Now, for $n\in\Nb$, pick $m_n\ge N_n$ and let
    \begin{align*}
        \mu_n\in\argmax_{\mu\in D_{\bar\mu}}\;(V^{m_n}-\varphi)(\mu). 
    \end{align*}
    Since $(V^{m_n}-\varphi)(\bar\mu)\ge 0$, it follows that $\mu_n\in B_{\varepsilon_n}(\bar\mu)$.
    We note that 
    \begin{align*}
        V^{m_n}(\mu_n)-V^+(\bar\mu)
        &\ge \varphi(\mu_n)-\varphi(\bar\mu)
        \ge-1/n,
    \end{align*}
    which completes the proof.
\end{proof}

Theorem~\ref{thm:dpp} together with Lemma~\ref{lem:pick_mu_n_continuity} immediately gives the following version of the DPP (note that the full strength of the lemma is in fact not used); we here focus on the inequality which will be needed for proving the subsolution property below and note that this inequality relies on the `deep inequality' of the DPP:

\begin{corollary}\label{cor:DPP_weak}
    Let $\bar\mu\in\Pc$ and $\varphi\in C^2(\Pc)$ such that $\varphi(\bar\mu)=V^+(\bar\mu)$ and $V^+-\varphi<0$ on $D_{\bar\mu}\setminus\{\bar\mu\}$. Then there exist sequences $m_n\to\infty$, $\mu_n\to\bar\mu$, $\mu_n\in D_{\bar\mu}$, such that     
    $$
        \varphi(\mu_n)\le \inf_{v\in\Uc}
        \E\left[\int_0^{\delta_{m_n}}e^{-\beta s}k\left(\xi^{v;\mu_n}_s,v\right)\ds+e^{-\beta\delta_{m_n}}\varphi\left(\xi^{v;\mu_n}_{\delta_{m_n}}\right)\right].
        %-\left(1-e^{-\beta\delta_{m_n}}\right)\kappa_n.
    $$     
\end{corollary}

Next we establish a certain `consistency property': 

\begin{proposition}\label{prop:consistency}
    Let $\bar\mu\in\Pc$. For any $\varphi\in C^2(\Pc)$, 
    \begin{align*}
        \liminf_{\substack{h\downarrow 0\\\mu\to\bar\mu}}
        \frac{1}{h}\bigg\{\varphi(\mu)-
        \inf_{v\in\Uc}\E\bigg[\int_0^h e^{-\beta s}k\left(\xi^{v;\mu}_s,v\right)\ds & +e^{-\beta h}\varphi\left(\xi^{v;\mu}_h\right)\bigg]\bigg\}\\
        &\ge
        H\left(\bar\mu,\varphi(\bar\mu),\piann{\varphi}{\mu}(\cdot,\cdot;\bar\mu)\right).
    \end{align*}
\end{proposition}

\begin{proof}
Let $\mu\in\Pc$, $v\in\Uc$ and $\varphi\in C^2(\Pc)$. By Lemma~\ref{lem:U}, $h(v,\cdot)$ is bounded on $\supp(\mu_0)$. Hence, for any $f\in C_b$, we have $\P\otimes\dt$-a.e.,
\begin{align}\label{eq:ito_ass}
    \int_0^t\left(\int_\R f(x)
    \left|h(v,x)-\xi^{v;\mu}_s(h(v,\cdot))\right|
    \xi^{v;\mu}_s(\dx)
    \right)^2\ds<\infty.
\end{align}
We may thus apply the It{\^o} formula \cite[Theorem~5.1]{cox_2024} to obtain, for $t\ge 0$, 
\begin{align}\label{eq:ito_appl}
    e^{-\beta t}\varphi\left(\xi^{v;\mu}_{t}\right)-\varphi(\mu)
    =&\int_0^te^{-\beta s}\left(-\beta\varphi(\xi^{v;\mu}_s)+L\varphi(\xi^{v;\mu}_s,v)\right)\ds\nonumber\\
    &+\int_0^t e^{-\beta s}\int_\R\pian{\varphi}{\mu}(x;\xi^{v;\mu}_s)\sigma(v,\xi^{v;\mu}_s;\dx)\di W_s,
\end{align}
where $\sigma(v,\mu;\dx)$ is given by \eqref{eq:def_sigma} and
\begin{align*}
    L\varphi(\mu,v)
    = 
    \frac{1}{2}\int_{\R\times\R}\piann{\varphi}{\mu}(x,y;\mu)
    \sigma(v,\mu;\dx)
    \sigma(v,\mu;\dy).
\end{align*}
Making again use of \eqref{eq:ito_ass}, we see that the last integral in \eqref{eq:ito_appl} is in fact a martingale; we thus obtain 
\begin{align}\label{eq:consistency_dynkin}
     e^{-\beta t}\E\left[\varphi\left(\xi^{v;\mu}_t\right)\right]-\varphi(\mu)
     =
     \E\int_0^t e^{-\beta s}\Big(-\beta\varphi(\xi^{v;\mu}_s)+L\varphi(\xi^{v;\mu}_s,v)\Big)\ds,
     \quad t\ge 0. 
\end{align}

Next, consider sequences $h_k\downarrow 0$ and $\mu_k\to\mu$; without loss of generality, suppose that $h_k\le 1$, $k\in\Nb$. For $v\in\Uc$, by Lemma~\ref{lem:U}, $h(v,\cdot)$ is continuous on $\supp(\mu_0)$ and therefore $L\varphi(\cdot,v)$ is continuous and thus admits a concave and continuous modulus of continuity, say $w$, on the compact set $\Pc$. Moreover, $\Wc(\xi^{v;\mu_k},\xi^{v;\mu})$ is a submartingale for $k\in\Nb$ (\cf{} \cite[Proposition~4.1.(i)]{beiglbock_2017}). Hence,
\begin{align}\label{eq:consistency_convergence}
    \frac{1}{h_k}\E\bigg[\int_0^{h_k}\Big|L\varphi(\xi^{v;\mu_k}_s,v)-L\varphi(\xi^{v;\mu}_s,v)\Big|\ds\bigg]
    \le 
    w\left(\E\;\Wc\left(\xi^{v;\mu_k}_1,\xi^{v;\mu}_1\right)\right)
    \xrightarrow[n\to\infty]{}0,
\end{align}
where the convergence follows from Theorem~\ref{thm:cont_MVM_poly}. 
Indeed, according to this result, for any subsequence of $\mu_k$, there exists a further subsequence, which we still index by $k$, along which $\Wc(\xi^{v;\mu_k}_1,\xi^{v;\mu}_1)\to 0$ \as{}, and along which the right-hand side of \eqref{eq:consistency_convergence} thus converges to zero; since any subsequence of the original sequence admits a further subsequence with this property, the convergence must hold also along the original sequence.
The same argument applies when replacing $L\varphi(\cdot,v)$ by $\varphi$ and $k(\cdot,v)$, making use of the standing Assumption~\ref{ass:U}.

Finally, for $\mu\in\Pc$, define 
\begin{align*}
    \Tc_h\varphi(\mu)
    :=\inf_{v\in\Uc}
        \E\bigg[\int_0^h e^{-\beta s}k\left(\xi^{v;\mu}_s,v\right) & \ds+e^{-\beta h}\varphi\left(\xi^{v;\mu}_h\right)\bigg].
\end{align*}
Fix now $\bar\mu\in\Pc$ and let $h_k\downarrow 0$ and $\mu_k\to\bar\mu$ such that
\begin{align*}
    \liminf_{\substack{h\downarrow 0\\\mu\to\bar\mu}}
    \frac{1}{h}\big(\varphi(\mu)-\Tc_h\varphi(\mu)\big)
    =
    \liminf_{k\to\infty}\frac{1}{h_k}\big(\varphi(\mu_k)-\Tc_{h_k}\varphi(\mu_k)\big).
\end{align*}
Let now also $v\in\Uc$ be fixed.
Making use of \eqref{eq:consistency_dynkin} and \eqref{eq:consistency_convergence}, we then obtain
\begin{align*}
    \liminf_{k\to\infty} & \frac{1}{h_k}\big(\varphi(\mu_k)-\Tc_{h_k}\varphi(\mu_k)\big)\\
    &\ge \liminf_{k\to\infty}\frac{1}{h_k}\left\{\varphi(\mu_k)-\E\bigg[\int_0^{h_k}e^{-\beta s}k\left(\xi^{v;\mu_k}_s,v\right)\ds+e^{-\beta h_k}\varphi\left(\xi^{ v;\mu_k}_{h_k}\right)\bigg]\right\}\\
    &= \liminf_{k\to\infty}
    \frac{1}{h_k}\E\bigg[\int_0^{h_k}e^{-\beta s}\Big(\beta\varphi(\xi^{v;\mu_k}_s)-k\left(\xi^{v;\mu_k}_s,v\right)-L\varphi(\xi^{v;\mu_k}_s,v)\Big)\ds\bigg]\\
    &= \liminf_{k\to\infty}
    \frac{1}{h_k}\E\bigg[\int_0^{h_k}e^{-\beta s}\Big(\beta\varphi(\xi^{v;\bar\mu}_s)-k\left(\xi^{v;\bar\mu}_s,v\right)-L\varphi(\xi^{v;\bar\mu}_s,v)\Big)\ds\bigg]\\
    &=
    \beta\varphi(\bar\mu)-k\left(\bar\mu,v\right)-L\varphi(\bar\mu,v),
    \end{align*}
    where, in order to obtain the last equality, we argued convergence by relying on the fundamental theorem of calculus for the inner integral and dominated convergence for the expected value. 
    Since $v\in\Uc$ was arbitrarily chosen, we may now take the supremum over $v\in\Uc$ to conclude.
\end{proof}

    We are now ready to establish that the limit of the auxiliary problem is a subsolution of the HJB equation.
    
\begin{theorem}\label{thm:V+_visc}
        The function $V^+$ is a viscosity subsolution of \eqref{eq:hjb}. 
    \end{theorem}

\begin{proof}
For $\bar\mu\in\Pc^s$, the subsolution property reduces to $\beta V^+(\bar\mu)\le\inf_{v\in\Uc}k(\bar\mu,v)$, which for Dirac measures holds by the definition of $V^+$; \cf{} Remark~\ref{rem:boundary}.
Let $\bar\mu\in\Pc\setminus\Pc^s$ and consider $\varphi\in C^2(\Pc)$ such that $\varphi(\bar\mu)=V^+(\bar\mu)$ and $V^+-\varphi<0$ on $D_{\bar\mu}\setminus\{\bar\mu\}$.
By Corollary~\ref{cor:DPP_weak}, we can then find sequences $m_n\to\infty$ and $\mu_n\to\bar\mu$, such that 
\begin{align*}
        0
        &\ge
        \liminf_{n\to\infty}
        \frac{1}{\delta_{m_n}}\left\{
        \varphi(\mu_n)
        -\inf_{v\in\Uc}
        \E\left[\int_0^{\delta_{m_n}}e^{-\beta s}k\left(\xi^{v;\mu_n}_s,v\right)\ds+e^{-\beta\delta_{m_n}}\varphi\left(\xi^{v;\mu_n}_{\delta_{m_n}}\right)\right]\right\}\\
        &\ge
        \liminf_{\substack{h\downarrow 0\\\mu\to\bar\mu}}
        \frac{1}{h}\bigg\{\varphi(\mu)-\inf_{v\in\Uc}
        \E\left[\int_0^h e^{-\beta s}k\left(\xi^{v;\mu}_s,v\right)\ds+e^{-\beta h}\varphi\left(\xi^{v;\mu}_h\right)\right]\bigg\}\\
        &\ge
        H\left(\bar\mu,\varphi(\bar\mu),\piann{\varphi}{\mu}(\cdot,\cdot;\bar\mu)\right),
\end{align*}
where the last inequality follows from Proposition~\ref{prop:consistency}. According to \eqref{eq:subsolution_alternative_continuity}, this completes the proof. 
\end{proof}

    \subsection{Proof of main result}\label{sec:proof_main}

    We first recall two results which ensure that the weak value function is a viscosity solution of the HJB equation and that we do have comparison for this equation.
    To this end, given $N$ points $x_1,\dots,x_N\in\supp(\mu_0)$, note that any given function $u:\Pc\to\R$ induces a function $\tilde u:\Delta^{N-1}\to\R$ defined by 
    \begin{align}\label{eq:reduction}
        \tilde u(p_1,\dots,p_N)
        = 
        u(p_1\delta_{x_1}+\dots+p_n\delta_{x_n}).
    \end{align}
    Moreover, by Lemma~\ref{lem:U}, $\{h(v,x):v\in\Uc,x\in\supp(\mu_0)\}$ is a bounded subset of $\R$ and $h(v,\cdot)$ is continuous on $\supp(\mu_0)$ for $v\in\Uc$. By the standing Assumption~\ref{ass:U}, we also have that $\mu\mapsto k(\mu,v)$ is continuous on $\Pc$ uniformly in $v\in\Uc$. 
    Hence, Assumptions (i)--(iii) of \cite[Theorem~6.2]{cox_2024} and (i) and (ii) of \cite[Theorem~9.1]{cox_2024} are satisfied:

    \begin{theorem}[Theorem~6.2 in \cite{cox_2024}]\label{thm:Vweak_visc}
        The function $V^{weak}$ is a viscosity solution of \eqref{eq:hjb}.
    \end{theorem}

    \begin{theorem}[Theorem~9.1 in \cite{cox_2024}]\label{thm:comparison}
    Let $u$ and $v$ be viscosity sub- and supersolutions, respectively, of \eqref{eq:hjb}.
    Suppose the following hold:
    \begin{enumerate}
        \item for any $N$ points $x_1,\dots,x_N\in\supp(\mu_0)$, the functions $\tilde u$ and $\tilde v$ obtained from $u$ and $v$ via \eqref{eq:reduction} are, respectively, upper and lower semicontinuous on $\Delta^{N-1}$;
        \item for any $\mu\in\Pc$, there exists a sequence of finitely supported $\mu_n\in\Pc$ such that
    \begin{align*}
        u(\mu)\le\liminf_{n\to\infty}u(\mu_n)
        \quad\textrm{and}\quad
        v(\mu)\ge\limsup_{n\to\infty}v(\mu_n).
    \end{align*}
    \end{enumerate}
    Then, $u\le v$. 
    In particular, \eqref{eq:hjb} admits a unique continuous viscosity solution. 
    \end{theorem}

    The formulation of this result is slightly different to that stated in \cite{cox_2024} since we impose somewhat weaker continuity assumptions on $u,v$. In addition, we work on a compact domain. We justify these modifications as follows:
    Note that \cite[Lemma~9.2 and 9.3]{cox_2024} still hold when relaxing their continuity assumptions on $\tilde u,\tilde v$ to only imposing our (i). Moreover, by our (ii),
    $$
    (u-v)(\mu)\le \liminf_{n\to\infty}(u-v)(\mu_n). 
    $$
    The theorem stated above can then be deduced along the same lines as \cite[Theorem~9.1]{cox_2024} (making obvious modifications to account for the fact that our $\Pc$ is a subset of the space considered therein).

    We recall that $V^+$ is upper semicontinuous on $\Pc$; we now
    argue that it is even continuous. Indeed, note that
    \begin{align}\label{eq:cont_V_+}
    \left|V^+(\mu)-V^+(\nu)\right|
    \le
    \sup_{n\in\Nb}\left|V^{n}(\mu)-V^{n}(\nu)\right|
    \le
    \sup_{n\in\Nb}\sup_{u\in\Ac^n}\big|\hat J(u;\mu)-\hat J(u;\nu)\big|.
    \end{align}
    It thus suffices to argue that the function $\hat J(u;\cdot)$ is continuous on $\Pc$ uniformly in $u\in\cup_{n\in\Nb}\Ac^n$; this follows however by use of the same arguments as used to prove Proposition~\ref{prop:properties_infinite}. In particular, $V^+$ satisfies the conditions on $u$ stated in Theorem~\ref{thm:comparison}. We next establish continuity properties of the weak value function:
    
    \begin{lemma}\label{lem:cont_V_weak}
        Assumptions (i) and (ii) of Theorem~\ref{thm:comparison} hold for $v=V^{weak}$.
    \end{lemma}

\begin{proof}
    (i)\;\;
    For a finitely supported $\mu\in\Pc$, according to Proposition~\ref{thm:cont_MVM_finite} (i), given a tuple $(\Omega,\Fc,\Fb,\P,W,u)$ such that $(\Omega,\Fc,\Fb,\P)$ is a filtered probability space, $W$ is an $(\Fb,\P)$-Brownian motion, and $u$ is an $\Uc $-valued and $\Fb$-progressively measurable process, there exists a pathwise unique solution to \eqref{eq:sde} equipped with the initial condition $\mu$. 
    Hence,
    \begin{align*}
        V^{weak}(\mu)=\inf_{(\Omega,\Fc,\Fb,\P,W,u)}\E\int_0^\infty e^{-\beta t}k(\xi^{u;\mu}_t,u_t)\dt,
    \end{align*}
    where the infimum is taken over all such tuples and the process $\xi^{u;\mu}$ denotes that associated solution to \eqref{eq:sde}.
    In consequence, for any two finitely supported measures $\mu,\nu\in\Pc$, 
    \begin{align*}
        \left|V^{weak}(\mu)-V^{weak}(\nu)\right|
        \le
        \sup_{(\Omega,\Fc,\Fb,\P,W,u)}\E\int_0^\infty e^{-\beta t}\left|k(\xi^{u;\mu}_t,u_t)-k(\xi^{u;\nu}_t,u_t)\right|\dt.
    \end{align*}
    Moreover, for any $\varepsilon>0$, there exists some $T>0$ such that, for any finitely supported measures $\mu,\nu\in\Pc$, and any tuple $(\Omega,\Fc,\Fb,\P,W,u)$,
    \begin{align*}
        \E\int_0^\infty e^{-\beta t}\left|k(\xi^{u;\mu}_t,u_t)-k(\xi^{u;\nu}_t,u_t)\right|\dt
    &\le 
        \E\int_0^T \left|k(\xi^{u;\mu}_t,u_t)-k(\xi^{u;\nu}_t,u_t)\right|\dt
        +\varepsilon\\
    &\le 
        Tw\left(\E\;\Wc\left(\xi^{u;\mu}_T,\xi^{u;\nu}_T\right)\right)
        +\varepsilon,
    \end{align*}
    where we used that, by the standing Assumption~\ref{ass:U}, $\mu\mapsto k(\mu,v)$ is continuous, uniformly in $v\in\Uc$, and thus admits a concave and continuous modulus of continuity $w$ on the compact set $\Pc$, and the fact that $\Wc(\xi^{u;\mu},\xi^{u;\nu})$ is a submartingale (\cf{} \cite[Proposition~4.1.(i)]{beiglbock_2017}).
    
    Restricting now to $\mu,\nu\in\Pc(\{x_1,\dots,x_N\})$ and applying Proposition~\ref{thm:cont_MVM_finite} (ii), we have that
    \begin{align*}
        \E\Wc\left(\xi^{u;\mu}_T,\xi^{u;\nu}_T\right)
    &\le 
        \max_{i\neq j}|x_i-x_j|\sqrt{N}\E\left\|((\xi^{u;\mu}_T-\xi^{u;\nu}_T)(\{x_1\}),\dots,(\xi^{u;\mu}_T-\xi^{u;\nu}_T)(\{x_N\}))\right\|\\
    &\le 
        \max_{i\neq j}|x_i-x_j|\sqrt{N}c\left\|((\mu-\nu)(\{x_1\}),\dots,(\mu-\nu)(\{x_1\}))\right\|,
    \end{align*}
    where $\|\cdot\|$ denotes the Euclidean norm and $c$ is a constant depending on $T$, $N$ and $\sup\{h(v,x):v\in\Uc,x\in\supp(\mu_0)\}$.
    Since, with the identification between $\alpha\in\Delta^{N-1}$ and $\mu\in\Pc(\{x_1,\dots,x_N\})$ given by $\mu=\alpha_1\delta_{x_1}+\dots+\alpha_N\delta_{x_N}$, the Euclidean and the Wasserstein distances are topologically equivalent, and since the obtained bound is independent of the chosen tuple, this yields continuity of $V^{weak}$ on $\Pc(\{x_1,\dots,x_N\})$.

    (ii)\;\;
    Let $\mu\in\Pc$ and let $\mu_k\in\Pc$ be a sequence of finitely supported measures such that $\Wc(\mu_k,\mu)\to 0$.
    For $\varepsilon>0$, let $(\Omega,\Fc,\Fb,\P,W,\xi,u)\in\Ac^{weak}(\mu)$ such that
    \begin{align*}
        \E\int_0^\infty e^{-\beta t}k(\xi_t,u_t)\dt
        \le V^{weak}(\mu)+\varepsilon/2.
    \end{align*}
    In turn, let $\xi^{u;\mu_k}$ be the solution to \eqref{eq:sde} constructed on the same probability space and with respect to the same control $u$ but for the initial condition $\mu_k$; it exists and is pathwise unique by Proposition~\ref{thm:cont_MVM_finite} (i). Then, there exists some $T>0$, such that
    \begin{align*}
        \limsup_{k\to\infty}V^{weak}(\mu_k)-V^{weak}(\mu)
        \le 
        \limsup_{k\to\infty}
        \E\int_0^T e^{-\beta t}\left(k(\xi^{u;\mu_k}_t,u_t)-k(\xi_t,u_t)\right)\dt+\varepsilon.
    \end{align*}
    Making use of Theorem~\ref{thm:cont_MVM_poly} and the standing Assumption~\ref{ass:U}, and arguing along the same lines as in the proof of Proposition~\ref{prop:properties_infinite}, we obtain that the limsup on the right-hand side must equal zero. Since $\varepsilon$ was arbitrarily chosen, this completes the proof.
    \end{proof}

    We are now ready to conclude the proof of our main result. 
   
    \begin{proof}[Proof of Theorem~\ref{thm:main}]
    Given $\varepsilon>0$, let $n\in\Nb$ such that $V^n(\mu_0)\le V^+(\mu_0)+\varepsilon$ and let $\hat u^n$ be the function provided by Corollary~\ref{cor:optimiser_existence} for which
    $
    \hat J(u^{*,n};\mu_0)\le V^n(\mu_0)+\varepsilon
    $
    when $u^{*,n}\in\Ac^n$ is defined by \eqref{eq:optimal_u} in terms of $\hat u^n$.
    Define, in turn, the feedback control $u^*$ by \eqref{eq:closing_constant} in terms of $\hat u^n$ (recall the relations \eqref{eq:y} and \eqref{eq:filter}); since the feedback function is piecewise constant, this control is well defined and belongs to $\mathcal{A}$. 
    
    We recall that $\pi^{u^*}_0=\mu_0$ and that $\pi^{u^*}$ satisfies
    \eqref{eq:pi}. Making again use of the fact that the control is
    piecewise constant, according to \cite[Theorem~2.1]{kunita_1971},
    we have pathwise uniqueness for the SDE \eqref{eq:pi} with $u$
    given by \eqref{eq:closing_constant} (viewing \eqref{eq:pi} as an
    equation for $\pi^{u^*}$ with $I^{u^*}$ a fixed
    $(\mathcal{Y}^{u^*},\P)$-Brownian motion)\footnote{Or, equivalently, for
    the SDE \eqref{eq:dynamics_controlled} with $u$ given by
    \eqref{eq:optimal_u}.}. We therefore also have uniqueness in law
    (\cf{} \cite[proof~of~Theorem V.3]{szpirglas_1978}) and it follows
    that $J(u^*)=\hat J(u^{*,n};\mu_0)$. In consequence,
    \begin{align*}
        J(u^*)\le V^+(\mu_0)+2\varepsilon.
    \end{align*}
    Since $\varepsilon>0$ was arbitrarily chosen, we obtain in particular that
    \begin{align}\label{eq:proof_main_ordering_mu0}
    V^{weak}(\mu_0) \le \inf_{u\in\Ac}J(u) \le V^+(\mu_0).
    \end{align} 

    Next, by Theorem~\ref{thm:V+_visc} and Theorem~\ref{thm:Vweak_visc}, $u=V^+$ and $v=V^{weak}$ are sub- and supersolutions, respectively, of \eqref{eq:hjb}.
    Moreover, by Lemma~\ref{lem:cont_V_weak} and the preceding discussion (\cf{} \eqref{eq:cont_V_+}), the continuity assumptions of Theorem~\ref{thm:comparison} hold for $u=V^+$ and $v=V^{weak}$. Application of this comparison result then yields $V^+\le V^{weak}$ on $\Pc$, which implies that both \eqref{eq:proof_main_ordering_continuity} and \eqref{eq:proof_main_ordering_mu0} hold as equalities. 
    Applying Theorem~\ref{thm:comparison} once again, making use of the continuity of $V^+$ (\cf{} \eqref{eq:cont_V_+}), yields the uniqueness and continuity claims.
    \end{proof}

\appendix

\section{Proof of the DPP}\label{sec:dpp}

We first argue that a certain `pseudo Markov property' holds within our context. To this end, let $n\in\Nb$ and recall the definition of $\delta_n$ from Section~\ref{sec:piecewise_constant}. Given a control $u\in\Ac^n$, we define a new control $u^{\delta_n,\omega}\in\Ac^n$, for $\omega\in\Omega$, by
    \begin{align*}
        u^{\delta_n,\omega}_t(\tilde\omega)
        =u_{\delta_n+t}(\omega\otimes_{\delta_n}\tilde\omega),
        \quad (t,\tilde\omega)\in[0,\infty)\times\Omega,
    \end{align*}
    where 
    $$
    (\omega\otimes_{\delta_n}\tilde\omega)(t)=\omega(t)\ind_{[0,\delta_n)}(t)+(\omega(\delta_n)+\tilde\omega(t-\delta_n)-\tilde\omega(0))\ind_{[\delta_n,\infty)}(t),
    \quad t\ge 0. 
    $$
The following result can then be deduced following the arguments in \cite{claisse_2016} and making use of the pathwise uniqueness of the flow solving \eqref{eq:dynamics_controlled} (the arguments simplify considerably since we here consider piecewise constant controls):

\begin{lemma}\label{lem:pseudo_markov}
    Let $n\in\Nb$ and $u\in\Ac^n$. Then, for every $\mu\in\Pc$ and $\P$-a.a. $\omega$,
    \begin{align*}
        \E\bigg[\int_{\delta_n}^\infty e^{-\beta(t-\delta_n)}k\left(\xi^{u;\mu}_t,u_t\right) & \dt\big|\Fc_{\delta_n}\bigg](\omega)\nonumber\\
        =
        \int_\Omega
        &\int_0^\infty e^{-\beta t}k\left(\xi^{u^{\delta_n,\omega};\xi^{u;\mu}_{\delta_n}(\omega)}_t(\tilde\omega),u^{\delta_n,\omega}_t(\tilde\omega)\right)\dt\;\di\P(\tilde\omega).
    \end{align*}
\end{lemma}

\begin{proof}[Proof of Theorem~\ref{thm:dpp}]
    `$\ge$': For any $\mu\in\Pc$ and $u\in\Ac^n$, by use of Lemma~\ref{lem:pseudo_markov} we obtain
    \begin{align}\label{eq:proof_dpp}
        \hat J(u;\mu)
        =&\;
        \E\left[\int_0^{\delta_n}e^{-\beta t}k\left(\xi^{u;\mu}_t,u_t\right)\dt+e^{-\beta\delta_n}\E\left[\int_{\delta_n}^\infty e^{-\beta(t-\delta_n)}k\left(\xi^{u;\mu}_t,u_t\right)\dt\big|\Fc_{\delta_n}\right]\right]\nonumber\\
        =&
        \int_\Omega \bigg\{\int_0^{\delta_n}e^{-\beta t}k\left(\xi^{u;\mu}_t(\omega),u_t(\omega)\right)\dt\nonumber\\
        &+e^{-\beta\delta_n}\int_\Omega\int_0^\infty e^{-\beta t}k\bigg(\xi^{u^{\delta_n,\omega};\xi^{u;\mu}_{\delta_n}(\omega)}_t(\tilde\omega),u^{\delta_n,\omega}_t(\tilde\omega)\bigg)\dt\;\di\P(\tilde\omega)\bigg\}\di\P(\omega)\\
        %\ge&
        %\int_\Omega \bigg\{\int_0^{\delta_n}e^{-\beta t}k\left(\xi^{u;\mu}_t(\omega),u_t(\omega)\right)\dt\\
        %&+e^{-\beta\delta_n}\inf_{\tilde u\in\Ac^n}\int_\Omega\int_0^\infty e^{-\beta t}k\bigg(\xi^{\tilde u,\xi^{u;\mu}_{\delta_n}(\omega)}_t(\tilde\omega),\tilde u_t(\tilde\omega)\bigg)\dt\;\di\P(\tilde\omega)\bigg\}\di\P(\omega)\\
        \ge&
        \int_\Omega \bigg\{\int_0^{\delta_n}e^{-\beta t}k\left(\xi^{u;\mu}_t(\omega),u_t(\omega)\right)\dt+e^{-\beta\delta_n}V^n\left(\xi^{u;\mu}_{\delta_n}(\omega)\right)\bigg\}\di\P(\omega)\nonumber\\
        =&\;
        \E\left[\int_0^{\delta_n}e^{-\beta t}k\left(\xi^{u;\mu}_t,u_t\right)\dt+e^{-\beta\delta_n}V^n\left(\xi^{u;\mu}_{\delta_n}\right)\right];\nonumber
     \end{align}
     taking the infimum over $u\in\Ac^n$ on both sides yields the desired inequality.

    `$\le$':
    Let $\varepsilon>0$. Thanks to the continuity properties established in Proposition~\ref{prop:properties_infinite}, following \cite{bouchard_2011}, we may apply an open covering argument to obtain, for some $N\in\Nb$, sets $A_i\subset\Pc$, $i=1,\cdots,N$, forming a partition of $\Pc$ and controls $u^i\in\Ac^n$ such that $\hat J(u^i;\mu)\le V^n(\mu)+\varepsilon$, for all $\mu\in A_i$, $i=1,\dots N$.
    Now, let $\mu\in\Pc$ and let $u\in\Ac^n$ be an arbitrary control. In turn, define a control $\hat u\in\Ac^n$ by
    \begin{equation*}
    \hat u_t(\omega)=\left\{\begin{array}{lll}
    u_t(\omega)&t\in[0,\delta_n)\\
    \sum_{i=1}^N\ind_{A_i}\left(\xi^{u;\mu}_{\delta_n}(\omega)\right)u^i_{t-\delta_n}\left(\theta_{\delta_n}\omega\right)&t\in[\delta_n,\infty)
    \end{array}\right.
    \quad (t,\omega)\in[0,\infty)\times\Omega,
    \end{equation*}
    where 
    $\theta_{\delta_n}\omega(t)
    =
    \omega(\delta_n+t)-\omega(\delta_n)$.
    Making once again use of Lemma~\ref{lem:pseudo_markov}, and the fact that, for each $i=1,\dots,N$ and $\omega\in\Omega$ such that $\xi^{u;\mu}_{\delta_n}(\omega)\in A_i$, it holds that 
    $\hat u^{\delta_n,\omega}_t(\tilde\omega)
    =
    u^i_t\left(\tilde\omega\right)$, for all
    $(t,\tilde\omega)\in[0,\infty)\times\Omega$, 
    we obtain for $\P$-a.a. $\omega$,
    \begin{align*}
       \E\bigg[\int_{\delta_n}^\infty & e^{-\beta(t-\delta_n)}k\left(\xi^{\hat u;\mu}_t,\hat u_t\right)\dt\big|\Fc_{\delta_n}\bigg](\omega)\\
        &=
        \int_\Omega\int_0^\infty e^{-\beta t}k\bigg(\xi^{\hat u^{\delta_n,\omega};\xi^{\hat u;\mu}_{\delta_n}(\omega)}_t(\tilde\omega),\hat u^{\delta_n,\omega}_t(\tilde\omega)\bigg)\dt\;\di\P(\tilde\omega)\\
        &=
        \sum_{i=1}^N\ind_{A_i}\left(\xi^{u;\mu}_{\delta_n}(\omega)\right)\left\{\int_\Omega\int_0^\infty e^{-\beta t}k\bigg(\xi^{u^i;\xi^{u;\mu}_{\delta_n}(\omega)}_t(\tilde\omega),u^i_t(\tilde\omega)\bigg)\dt\;\di\P(\tilde\omega)\right\}\\
        &\le
        \sum_{i=1}^N\ind_{A_i}\left(\xi^{u;\mu}_{\delta_n}(\omega)\right)\left\{V^n\left(\xi^{u;\mu}_{\delta_n}(\omega)\right)+\varepsilon\right\}
        = V^n\left(\xi^{u;\mu}_{\delta_n}\right)(\omega)+\varepsilon.
    \end{align*}
    In consequence,
     \begin{align*}
        V^n(\mu)
        \le&\;
        \hat J(\hat u;\mu)\\
        =&\;
        \E\left[\int_0^{\delta_n}e^{-\beta t}k\left(\xi^{\hat u;\mu}_t,\hat u_t\right)\dt+e^{-\beta\delta_n}\E\left[\int_{\delta_n}^\infty e^{-\beta(t-\delta_n)}k\left(\xi^{\hat u;\mu}_t,\hat u_t\right)\dt\big|\Fc_{\delta_n}\right]\right]\\
        \le&\;
        \E\left[\int_0^{\delta_n}e^{-\beta t}k\left(\xi^{u;\mu}_t,u_t\right)\dt+e^{-\beta\delta_n}V^n\left(\xi^{u;\mu}_{\delta_n}\right)\right]+e^{-\beta\delta_n}\varepsilon.
     \end{align*}
     Taking the infimum over $u\in\Ac^n$ and using that $\varepsilon>0$ was arbitrary, yields the claim. 
    \end{proof}

    \begin{proof}[Proof of Corollary~\ref{cor:optimiser_existence}]
    Let $\varepsilon>0$ and let $\{A_i\}_{i=1}^N\subset\Pc$ and $\{u^i\}_{i=1}^N\subset\Ac^n$ be defined as in the proof of Theorem~\ref{thm:dpp}. In turn, define $\hat u^n:\Pc\to\Uc$ by

    $$
    \hat u^n(\mu):=\sum_{i=1}^N\ind_{A_i}\left(\mu\right)u^i_0,
    \quad\mu\in\Pc.
    $$
    Note that, thanks to the properties of $u^i$, $i=1,\dots,N$, and \eqref{eq:proof_dpp}, for all $\mu\in\Pc$,
    \begin{align}\label{eq:proof_exist_control_opt_continuity}
        V^n(\mu)
        \ge
        \E\left[\int_0^{\delta_n}e^{-\beta t}k\left(\xi^{\hat u^n(\mu);\mu}_t,\hat u^n(\mu)\right)\dt+e^{-\beta\delta_n}V^n\left(\xi^{\hat u^n(\mu);\mu}_{\delta_n}\right)\right]-\varepsilon.
    \end{align}
    
    Let $\mu\in\Pc$ and define $u\equiv u^{*,n}\in\Ac^n$ by \eqref{eq:optimal_u} with respect to $\hat u^n$; it remains to argue that $u$ is `$\varepsilon$-optimal'. 
    Making use of \eqref{eq:proof_exist_control_opt_continuity} together with the fact that
    $
    \hat u^n(\xi^{u;\mu}_{\delta_n}(\omega))=u^{\delta_n,\omega}_0,
    $
    for each $\omega\in\Omega$, and Lemma~\ref{lem:pseudo_markov}, we obtain
    \begin{align*}
       V^n\left(\xi^{u;\mu}_{\delta_n}\right)(\omega)
        \ge&\;
       \int_\Omega \bigg\{\int_0^{\delta_n}e^{-\beta t}k\bigg(\xi^{\hat u^n\left(\xi^{u;\mu}_{\delta_n}(\omega)\right);\xi^{u;\mu}_{\delta_n}(\omega)}_t(\tilde\omega),\hat u^n\left(\xi^{u;\mu}_{\delta_n}(\omega)\right)\bigg)\dt\\
       &\qquad\;
       +e^{-\beta\delta_n}V^n\bigg(\xi^{\hat u^n\left(\xi^{u;\mu}_{\delta_n}(\omega)\right);\xi^{u;\mu}_{\delta_n}(\omega)}_{\delta_n}(\tilde\omega)\bigg)\bigg\}\di\P(\tilde\omega)-\varepsilon\\
       =&
       \int_\Omega \bigg\{\int_0^{\delta_n}e^{-\beta t}k\bigg(\xi^{u^{\delta_n,\omega};\xi^{u;\mu}_{\delta_n}(\omega)}_t(\tilde\omega),u^{\delta_n,\omega}_t(\tilde\omega)\bigg)\dt\\
       &\qquad\;
       +e^{-\beta\delta_n}V^n\bigg(\xi^{u^{\delta_n,\omega};\xi^{u;\mu}_{\delta_n}(\omega)}_{\delta_n}(\tilde\omega)\bigg)\bigg\}\di\P(\tilde\omega)-\varepsilon\\
       =&\;
       \E\bigg[\int_{\delta_n}^{2\delta_n}e^{-\beta (t-\delta_n)}k\left(\xi^{u;\mu}_t,u_t\right)\dt+e^{-\beta\delta_n}V^n\left(\xi^{u;\mu}_{2\delta_n}\right)\big|\Fc_{\delta_n}\bigg](\omega)-\varepsilon.
    \end{align*}
    Now, by definition, \eqref{eq:proof_exist_control_opt_continuity} holds also with $\hat u^n(\mu)$ replaced by $u$; substituting the above into this inequality, we obtain
    \begin{align*}
        V^n(\mu)
    \ge&\;
        \E\bigg[\int_0^{2\delta_n}e^{-\beta t}k\left(\xi^{u;\mu}_t,u_t\right)\dt+e^{-\beta2\delta_n}V^n\left(\xi^{u;\mu}_{2\delta_n}\right)\bigg]-\varepsilon \big(1+e^{-\beta\delta_n}\big)\\
    \ge&\;
        \E\left[\int_0^{l\delta_n}e^{-\beta t}k\left(\xi^{u;\mu}_t,u\right)\dt+e^{-\beta l\delta_n}V^n\left(\xi^{u;\mu}_{l\delta_n}\right)\right]
        -\varepsilon\big(1+\dots+e^{-\beta(l-1)\delta_n}\big),
        %-\varepsilon\big(1+e^{-\beta\delta_n}+\dots+e^{-\beta(l-1)\delta_n}\big),
     \end{align*}
     %\begin{align*}
     %   \E\bigg[\int_0^{(l-1)\delta_n}e^{-\beta t} & k\left(\xi^{u;\mu}_t,u_t\right)\dt+e^{-\beta(l-1)\delta_n}V^n\left(\xi^{u;\mu}_{(l-1)\delta_n}\right)\bigg]\\
     %   \ge&\;
     %   \E\bigg[\int_0^{l\delta_n}e^{-\beta t}k\left(\xi^{u;\mu}_t,u_t\right)\dt+e^{-\beta l\delta_n}V^n\left(\xi^{u;\mu}_{l\delta_n}\right)\bigg]-\varepsilon e^{-\beta(l-1)\delta_n}.
     %\end{align*}
    where the second inequality follows by iterated use of the same arguments and holds for any $l\in\Nb$.
    Since $V^n$ is bounded and the geometric series converges, we may conclude by sending $l$ to infinity.
    %\begin{align*}
    %    \E\int_0^\infty e^{-\beta t}k\left(\xi^{u;\mu}_t,u\right)\dt
    %    \le
    %    V^n(\mu)+\varepsilon(1-e^{-\beta\delta_n})^{-1},
    %\end{align*}
    \end{proof}

    \section{Technical justification of Example~\ref{ex:sensor}}\label{sec:example_h}

    In this section, we verify that the problem formulated in Example~\ref{ex:sensor} satisfies all the conditions required in Section~\ref{sec:formulation}.

    Due to the form of $\tilde h$, we expect to be able to rewrite $\tilde{h}(\tilde{v}, x) = \sum_{i=0}^{\infty} C_i(\tilde{v}) x^i$. Let $\Phi: I \to \mathbb{R}^{\mathbb{N}}$ be defined by $\Phi(\tilde{v}) = \big(C_0(\tilde{v}), C_1(\tilde{v}), \ldots\big)$, and set $\mathcal{U}:=\Phi(I)$. We show in Lemma~\ref{lem:example_h} that $\Phi$ is well-defined, injective, continuous, and $\sum_{i=0}^{\infty} R^{2i} C_{i}(\tilde{v})^2 \leq K$ for some $R>1$, $K>0$, which implies that  $\mathcal{U}$ is a compact subset of $\{v = (v_0, v_1,\ldots): v_i \in \mathbb{R}, i =0,1,\ldots, \sum_{i=0}^{\infty} R^{2i} v_i^2 \leq K\}$; in particular, $\Phi$ is a bijection between $I$ and $\mathcal{U}$. We then define a cost function $k: \mathcal{P} \times \mathcal{U} \to \mathbb{R}$ such that $k(\mu, \Phi(\tilde{v})) = \tilde{k}(\mu, \tilde{v})$ for all $\tilde{v} \in I$. We verify the continuity of $k$ in terms of $\mathcal{W}_1$: Since $e^{-\theta (x-\tilde{v})^2}$ is Lipschitz continuous in $x$ with a uniform Lipschitz constant over $\tilde{v}$, the fear term is continuous; an analogous argument applies to the reward term.

    It remains to prove the following lemma which was used above:
    \begin{lemma}\label{lem:example_h}      
        Let $R>1$ so that $I \subset (-R,R)$. The sensor function \eqref{eq:example h} satisfies, for some $K>0$, 
        \begin{equation}\label{eq:example h expansion}
        \tilde{h}(\tilde{v}, x) = \sum_{i=0}^{\infty} C_i(\tilde{v}) x^i, 
        \quad 
        \text{ where }
        \ 
        \sum_{i=0}^{\infty} \big(R^i C_i(\tilde{v})\big)^2 \leq K, \quad \text{for all } x, \tilde{v} \in I,  
        \end{equation}
        where $\tilde{v} \mapsto (C_0(\tilde{v}), C_1(\tilde{v}), \ldots)$ is an injective and continuous function from $I$ to $\mathbb{R}^{\mathbb{N}}$ in the product topology. 
    \end{lemma}
    \begin{proof}
        We can rewrite $\tilde{h}$ as 
        \begin{align*}
        \tilde{h}(\tilde{v}, x) 
        = 
        e^{-\lambda x^2} e^{2 \lambda \tilde{v} x} e^{-\lambda \tilde{v}^2}
        = 
        e^{-\lambda \tilde{v}^2} \sum_{m=0}^{\infty} \frac{(2 \lambda \tilde{v})^m}{m!} x^m \sum_{k=0}^{\infty} \frac{(-\lambda)^{k}}{k!}  x^{2k}  
        = 
        \sum_{i=0}^{\infty} C_i(\tilde{v}) x^i,
        \end{align*}
        where 
        $$
        C_i(\tilde{v}) = e^{-\lambda \tilde{v}^2} \sum_{j=0}^{\lfloor \frac{i}{2} \rfloor} \frac{(-\lambda)^j (2 \lambda \tilde{v})^{(i-2j)}}{j! (i-2j)!}. 
        $$
        That $\tilde{v} \mapsto (C_{0}(\tilde{v}),C_{1}(\tilde{v}),\ldots)$ is well-defined follows from the uniqueness of power series representations; that it is injective follows from the injectiveness of $\tilde{v} \mapsto \tilde{h}(\tilde{v},\cdot)$. It is evident from the above expression that each $C_i$ is continuous on $I$.
        Let $\wti{R}>R$ and define $M := e^{\lambda (\wti{R}+R)^2}$. Note that the same calculation gives 
        $$
        e^{\lambda (x+\tilde{v})^2} = \sum_{i=0}^{\infty} \wti{C}_i(\tilde{v}) x^i,
        \quad \text{ where }\ 
        \wti{C}_i(\tilde{v}) = e^{\lambda \tilde{v}^2} \sum_{j=0}^{\lfloor \frac{i}{2} \rfloor} \frac{\lambda^j (2 \lambda \tilde{v})^{(i-2j)}}{j! (i-2j)!}. 
        $$
        Hence, $\sum_{i=0}^{\infty} \wti{C}_i(\tilde{v}) \wti{R}^i = e^{\lambda (\wti{R}+\tilde{v})^2} \leq M$ for all $\tilde{v} \in I$. Since $\wti{C}_i(\tilde{v})$ and $\wti{R}$ are positive, we then have for all $i \in \mathbb{N}$ and all $\tilde{v} \in I$ that $|C_i(\tilde{v})| \leq \wti{C}_i(\tilde{v}) \leq \frac{M}{{\wti{R}}^i}$. Since $R<\wti{R}$, we conclude that
        $$
        \sum_{i=0}^{\infty} C_i(\tilde{v})^2 R^{2i}
        \leq 
        \sum_{i=0}^{\infty} \frac{M^2}{{\wti{R}}^{2i}} R^{2i}
        =
        \frac{M^2}{1 - \big(\frac{R}{\wti{R}}\big)^2}
        =: K,
        $$
        which completes the proof. 
    \end{proof}

\bibliography{ref}
\bibliographystyle{plain}

\end{document}